\documentclass[pdflatex,sn-mathphys-num]{sn-jnl}

\usepackage{graphicx}%
\usepackage{multirow}%
\usepackage{amsmath,amssymb,amsfonts}%
\usepackage{amsthm}%
\usepackage{mathrsfs}%
\usepackage[title]{appendix}%
\usepackage{xcolor}%
\usepackage{textcomp}%
\usepackage{manyfoot}%
\usepackage{booktabs}%
\usepackage{algorithm}%
\usepackage{algorithmicx}%
\usepackage{algpseudocode}%
\usepackage{listings}%

\usepackage{amstext}
\usepackage{afterpage}

\theoremstyle{thmstyleone}%
\theoremstyle{thmstyletwo}%
\newtheorem*{remark}{Remark}%
\newtheorem{lemma}{Lemma}

\theoremstyle{thmstylethree}%
\newtheorem*{claim}{Claim}

\begin{document}

\title[Multi-shooting parameterization methods for invariant manifolds and heteroclinics of 2 DOF Hamiltonian Poincar\'e maps, with applications to celestial resonant dynamics]{Multi-shooting parameterization methods for invariant manifolds and heteroclinics of 2 DOF Hamiltonian Poincar\'e maps, with applications to celestial resonant dynamics}

\author*[1]{\fnm{Bhanu} \sur{Kumar}}\email{bhkumar@umich.edu}
\affil[1]{\orgdiv{Department of Mathematics}, \orgname{University of Michigan}, \orgaddress{\street{530 Church St}, \city{Ann Arbor}, \state{MI}, \postcode{48109}, \country{USA}}}

\abstract{Studying 2 degree-of-freedom (DOF) Hamiltonian dynamical systems often involves the computation of stable \& unstable manifolds of periodic orbits, due to the homoclinic \& heteroclinic connections they can generate. Such study is generally facilitated by the use of a Poincar\'e section, on which the manifolds form 1D curves. A common method of computing such manifolds in the literature involves linear approximations of the manifolds, while the author's past work has developed a nonlinear manifold computation method under the assumption that the periodic orbit intersects the chosen Poincar\'e section only once. However, linear manifold approximations may require large amounts of numerical integration for globalization, while the single-intersection assumption of the previous nonlinear method often does not hold. 

In this paper, a parameterization method is developed and implemented for computing such stable and unstable manifolds even in the case of multiple periodic orbit intersections with a chosen Poincar\'e section. The method developed avoids the need to compose polynomials with Poincar\'e maps -- a requirement of some previous related algorithms -- by using an intermediate step involving fixed-time maps. The step yields curves near the chosen Poincar\'e section lying on the flow's periodic orbit manifolds, which are used to parameterize and compute the Poincar\'e map manifold curves themselves. These last curves and parameterizations in turn enable highly-accurate computation of heteroclinics between periodic orbits. 
The method has already been used for various studies of resonant dynamics in the planar circular restricted 3-body problem, which are briefly summarized in this paper, demonstrating the algorithm's utility for real-world investigations. }


\maketitle

\section{Introduction}


Given a 2 DOF Hamiltonian system, oftentimes its dynamical analysis requires computation of stable and unstable manifolds of unstable periodic orbits, along with detection of the manifolds' intersections resulting in heteroclinic or homoclinic connections. Various methods have been developed for computing the stable/unstable manifolds of periodic orbits in such systems, e.g. the flow-based Chebyshev-Taylor parameterization methods of \cite{chebTaylor}. Parameterization methods \cite{CabreFontichLlave, haroetal} are a general class of nonlinear methods for computing various types of invariant objects in dynamical systems. They allow computing local stable/unstable manifolds significantly more accurately than the (linear) monodromy-matrix eigenvector approximations commonly used in the literature, e.g. \cite{Anderson2010, Anderson2011, vaqueroHowell}. The resulting representations of the manifolds are accurate in a  much larger neighborhood of the base periodic orbit, as compared to linear manifold approximations. This yields local manifolds requiring less numerical integration to globalize. 

If one fixes the Hamiltonian energy and reduces the dynamics to a 2D Poincar\'e section, the stable/unstable manifolds of any periodic orbit will be 1D curves emanating from the intersection points of the base periodic orbit with the section; the number of such intersection points will naturally depend on the section chosen. If a periodic orbit intersects a given section only once, the orbit will correspond to a fixed point of the Poincar\'e map; for this case, a parameterization method for computing accurate Taylor series parameterizations of the manifolds was presented in \cite{kumar2021journal}, yielding a single 1D curve on the Poincar\'e section for each of the stable and unstable manifolds. In applications, oftentimes a section is chosen such that some periodic orbit of interest only intersects the section at a single point. However, such sections may not have good transversality properties, in the sense that this section may have tangencies to the Hamiltonian vector field at many points -- including at points of the global stable/unstable manifolds of the original periodic orbit. This in turn can introduce discontinuities in the computed manifold curves when globalized.

The aforementioned problem of discontinuities is illustrated most clearly by a number of studies carried out on the well-known planar circular restricted 3-body problem model (PCRTBP). The PCRTBP is a 2 DOF Hamiltonian system which played an important role in the development of Hamiltonian mechanics, and is still important for applications today. It approximates the motion of a spacecraft or other small celestial body under the gravitational influence of a planet-moon or star-planet pair, such as Earth-Moon or Sun-Earth. The Poincar\'e section used in most past PCRTBP studies (e.g. \cite{Anderson2010, Anderson2011, vaqueroHowell, KoLoMaRo}) has been a fixed $x$ or fixed $y$ section, often with some additional constraint on the signs of $x$ and $\dot y$. For example, Anderson and Lo \cite{Anderson2010} use a $y = 0$, $x < 0$, $\dot y > 0$ Poincar\'e surface of section to study a class of PCRTBP periodic orbits known as \emph{unstable resonant orbits}; these orbits intersect the aforementioned Poincar\'e section only once. However, in \cite{kumar2021journal}, we found that the manifolds of such resonant orbits experience tangencies to that section, resulting in manifold curves with discontinuous jumps in several places. 

A way to avoid experiencing such discontinuities in the computed manifold curves is to choose a different Poincar\'e section with better transversality to the flow in the region of interest, so that the manifolds no longer have points of tangency to the section in that region. This, however, can require choosing a section having multiple intersection points with each periodic orbit of interest. For example, in the PCRTBP, there is a quantity $\nu$ known as \emph{osculating true anomaly} that, in most of the phase space (including the unstable resonant orbits), is always increasing; thus, a Poincar\'e section formed by fixing $\nu$ will be transverse to the PCRTBP flow at all points in this phase space region. However, most unstable resonant periodic orbits intersect this fixed-$\nu$ section not only at a single point, but at multiple points. Thus, while this Poincar\'e section is better for dynamical analysis, the accurate stable/unstable manifold parameterization methods and analysis techniques described in our previous paper \cite{kumar2021journal} no longer work with this section. 

If one chooses a surface of section such that the periodic orbits of interest intersect the section multiple times, then the periodic orbit of the flow will also be a periodic orbit of the corresponding Poincar\'e map. Parameterization methods for stable/unstable periodic orbits of maps were developed in \cite{gonzalezJames}, under the assumption that the map has a form that can be composed with Taylor series to yield another Taylor series (a step which needs to be computationally implemented). However, in our case, composing a Poincar\'e map with Taylor series is quite a difficult step, since the map is defined by propagating a system of differential equations rather than the closed-form algebraic expressions considered in \cite{gonzalezJames}. While it is theoretically possible to computationally apply a Poincar\'e map to Taylor series (e.g. \cite{perezthesis, gimeno2023}), this step greatly increases the complexity of the algorithms, so it would be beneficial to avoid the aforementioned computation. 

In this paper, we develop accurate methods for computing stable and unstable manifolds of periodic orbits in 2 DOF Hamiltonian Poincar\'e maps, as well as heteroclinic connections between them. This is all carried out with a view towards applications in the PCRTBP to finding transitions between resonant orbits. The methods developed for computing the manifolds involves a multi-shooting type parameterization method that yields high-order polynomial approximations for curves of stable/unstable manifold points, while avoiding applying a Poincar\'e map to Taylor series. Finally, we are able to combine the computed polynomials with a Poincar\'e section to detect and compute heteroclinic connections. We have successfully implemented and used these tools already in a number of studies \cite{rawat2025preprint, kumar2024aug, kumar2024iac} on resonant dynamics in various celestial systems, which we also briefly review as a demonstration of the methods' efficacy. 

\section{Background and Models} \label{backgroundsection}

In this section, we give some background on the parameterization method for invariant manifolds as well as on the PCRTBP and mean motion resonant orbits. Readers familiar with the parameterization method may skip Section \ref{paramethodsection}; readers familiar with the PCRTBP and with osculating orbital elements may skip Section \ref{modelsection}. Section \ref{mmrsection} briefly discusses mean motion resonances and their overlap. 

\subsection{The Parameterization Method for Invariant Manifolds} \label{paramethodsection}

The following standard description of the parameterization method is reproduced identically from our previous paper \cite{kumar2021journal}. The parameterization method is a technique in dynamical systems useful for the computation of several types of invariant geometric objects, including invariant tori as well as stable and unstable manifolds of fixed points, periodic orbits, and tori. It works in both Hamiltonian as well as non-Hamiltonian systems. Haro et al. \cite{haroetal} provide an excellent reference for many applications of this method. The essential idea is that if one has a map $F: M \rightarrow M$ where $M$ is some manifold, and knows that there is an $F$-invariant object diffeomorphic to some model manifold $\mathcal{K}$, then one can solve for an injective immersion $W:\mathcal K \rightarrow M$ and a diffeomorphism $f: \mathcal K \rightarrow \mathcal K$ such that the invariance equation
\begin{equation}  \label{invariancequation}   F(W(s)) = W(f(s)) \end{equation}
holds for all $s \in \mathcal K$. $W$ is called the parameterization of the invariant manifold; Equation \eqref{invariancequation} simply states that $F$ maps the image $W(\mathcal K)$ into itself, so that $W(\mathcal K)$ is the invariant object in the full ambient manifold $M$. The dynamics inside $W(\mathcal K)$ are then conjugate to $f$, their representation on the model manifold $\mathcal K$.

\subsection{Planar Circular Restricted 3-Body Problem} \label{modelsection}
The PCRTBP \cite{kumar2022} describes the motion of a small particle (e.g. a spacecraft) under the gravitational influence of two much larger bodies of masses $m_{1}$ and $m_{2}$, which revolve about their barycenter in a circular orbit (e.g. a planet and moon). Time, length, and mass units are normalized so that the distance between $m_{1}$ and $m_{2}$ becomes 1, their period of revolution becomes $2 \pi$, and $\mathcal{G}(m_{1}+m_{2})$ becomes 1 (where $\mathcal G$ is the gravitational constant). Defining a mass ratio $\mu = \frac{m_{2}}{m_{1}+ m_{2}}$, one uses a rotating non-inertial Cartesian coordinate system centered at the $m_{1}$-$m_{2}$ barycenter such that $m_{1}$ and $m_{2}$ are always on the $x$-axis, at  $(-\mu,0)$ and  $(1-\mu,0)$ respectively. In the planar CRTBP, we also assume that the spacecraft moves in the same plane as $m_{1}$ and $m_{2}$. Then, the equations of motion are 2 degree-of-freedom (DOF) Hamiltonian \cite{celletti}:

\begin{equation} \label{pcrtbpH_EOM} \dot x = \frac{\partial H}{\partial p_{x}} \quad \dot y = \frac{\partial H}{\partial p_{y}} \quad \quad \dot p_{x} = -\frac{\partial H}{\partial x} \quad \dot p_{y} = -\frac{\partial H}{\partial y} \end{equation}

\begin{equation}  \label{pcrtbpH} H(x,y,p_x,p_{y})= \frac{p_{x}^{2}+p_{y}^{2}}{2} + p_{x}y -p_{y}x - \frac{1-\mu}{r_{1}} - \frac{\mu}{r_{2}} \end{equation}
where $r_{1} = \sqrt{(x+\mu)^{2} + y^{2}}$ and $r_{2} = \sqrt{(x-1+\mu)^{2} + y^{2}} $ are the distances from the spacecraft to $m_{1}$ and $m_{2}$ respectively. The momenta $p_{x}, p_{y}$ are the spacecraft velocity components in an inertial reference frame; they are related to the rotating (non-inertial) frame velocities $\dot x, \dot y$ as $ p_{x} =  \dot x - y$, $p_{y} = \dot y + x$. 

There are two important properties of Eq. \eqref{pcrtbpH_EOM}-\eqref{pcrtbpH} to note. First of all, the Hamiltonian in Eq. \eqref{pcrtbpH} is autonomous and is thus an integral of motion. Hence, PCRTBP trajectories are restricted to 3D energy submanifolds of the state space satisfying $H(x,y,p_{x}, p_{y})=$ constant. The quantity $C = -2H$ is referred to as the \emph{Jacobi constant}, and is generally used in lieu of $H$ to specify energy levels. The second property is that the equations of motion have a time-reversal symmetry. Namely, if $(x(t), y(t), p_{x}(t), p_{y}(t), t)$ is a solution of Eq. \eqref{pcrtbpH_EOM}-\eqref{pcrtbpH} for $t > 0$, then $(x(-t), -y(-t), -p_{x}(-t), p_{y}(-t), t)$ is a solution for $t < 0$. 

\subsubsection{True anomaly and synodic Delaunay coordinates} \label{delaunaySection}

The PCRTBP for $\mu = 0$ is simply the 2-body (Kepler) problem in a rotating coordinate frame, since this corresponds to $m_{2} = 0$. In the Kepler problem written in an \emph{inertial} (non-rotating) reference frame, it is known that all bounded orbits form ellipses with focus at $m_{1}$. One can thus transform to non-Cartesian coordinates known as \emph{orbital elements} which describe the size, shape, and orientation of said ellipse, as well as the position of the spacecraft on the ellipse at a given time. This last property is described by the \emph{true anomaly} $\nu \in \mathbb{T}$, defined as the angle formed by spacecraft position, $m_{1}$, and the orbit ellipse periapse (point of closest approach to $m_{1}$ over the orbit). When considering only planar orbits restricted to the $xy$-plane, there will be 3 other orbital elements: the orbital ellipse's semimajor axis $a$, the ellipse's eccentricity $e$, and its \emph{longitude of periapse} $g_{0} \in \mathbb{T}$ -  the angle between the orbital ellipse's periapse, $m_{1}$, and the positive $x$-axis. These 3 elements $a$, $e$, and $g_{0}$ remain constant for all time in the Kepler problem. For more details on these standard definitions and transformations to/from Cartesian coordinates, we refer the reader to \cite{bmw}. 

While the above discussion used a non-rotating inertial reference frame, one can easily transform rotating frame positions and velocities to an inertial frame; in fact, the momenta $p_{x}$, $p_{y}$ defined in Section \ref{modelsection} \emph{are} velocities with respect to the non-rotating coordinate frame whose $x$ and $y$ axes are aligned with those of the rotating frame at a given instant\footnote{This means that we consider a series of \emph{different} inertial frames, one at each time, all of which are related to each other through rotations. This time-variation of the inertial frame used for each transformation has no effect on $a$, $e$, and $\nu$. The only orbital element changed is the argument of periapse, which was previously defined relative to the $x$-axis of a single inertial frame, but now is being computed relative to a time-varying series of different inertial frames' $x$-axes (aligned with the rotating frame for each time). The notation $g$ reflects its difference from $g_{0}$ which was defined relative a single fixed inertial frame}. Thus, after a shift of origin to $m_{1}$, $(x,y,p_{x}, p_{y})$ can be transformed into orbital elements by the same equations used in the inertial frame Kepler problem. Moreover, this transformation, being defined through fixed functions of $(x,y,p_{x}, p_{y})$, is also valid for $\mu > 0$ in the PCRTBP, yielding what are known as \emph{osculating} orbital elements $(a, e, g, \nu)$ -- the orbital elements of the elliptical orbit that would occur if $m_{2}$ suddenly vanished, as illustrated in Figure \ref{fig:osculatingEls}.
\begin{figure}
\centering
\includegraphics[width=0.5\textwidth]{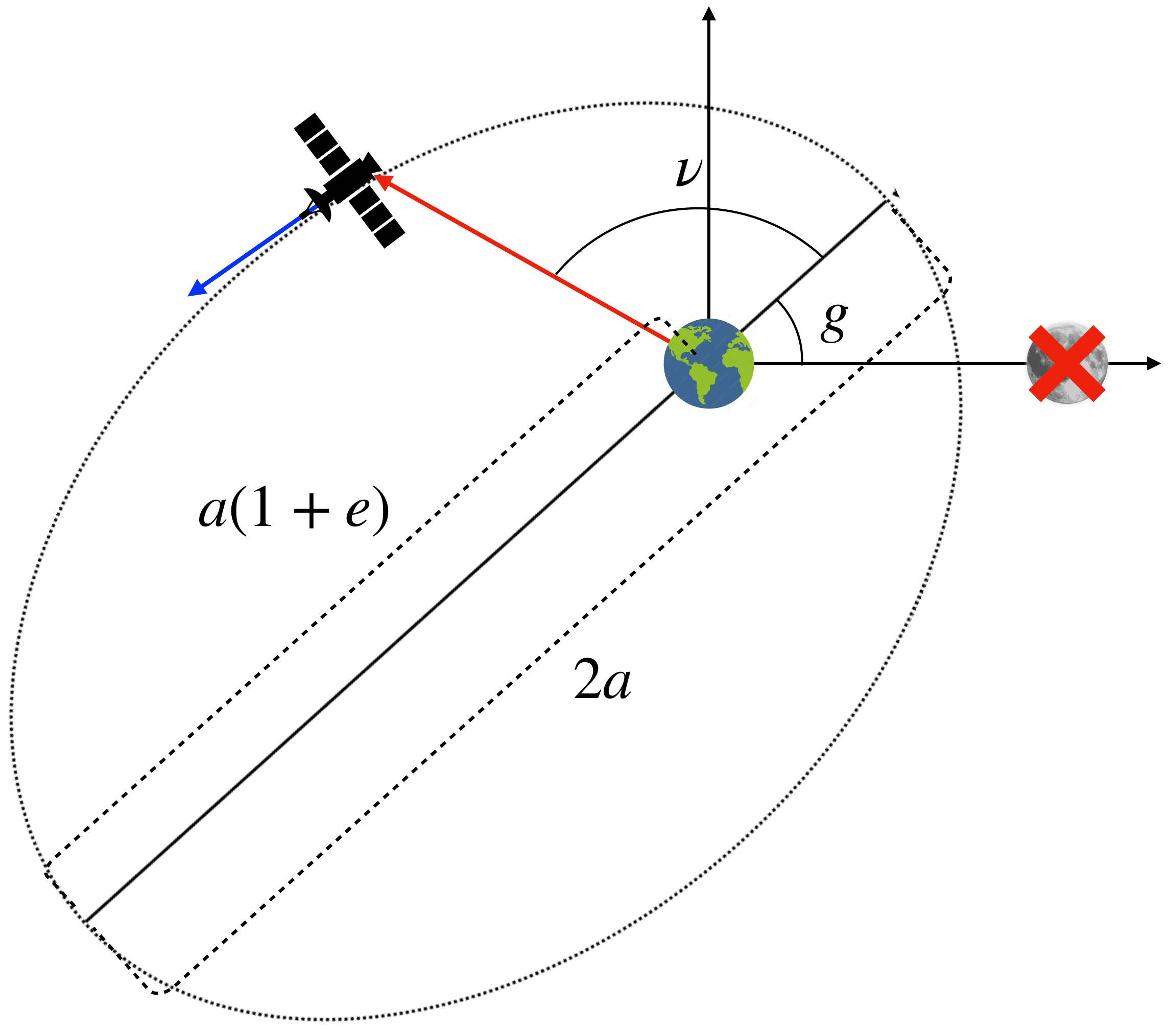}
\caption{  \label{fig:osculatingEls}Illustration of osculating orbital elements $(a,e,g,\nu)$} 
\end{figure} 
Here, $a$, $e$, and $\nu$ have the same geometric interpretation as the inertial frame Kepler problem, while $g = g_{0} - t$ is the longitude of periapse \emph{with respect to the rotating frame x-axis} (see footnote). See Celletti \cite{celletti} for more details on $g$. 

The key property of $\nu$ is that in the $\mu = 0$ Kepler problem, $\dot \nu > 0$ for all time and points in the phase space. When $\mu > 0$, it is no longer guaranteed that $\dot \nu > 0$ in the entire phase space, but in practice this is still the case except for in a small region near $m_{2}$. Thus, $\nu$ makes an ideal variable for defining Poincar\'e sections outside of a neighborhood of $m_{2}$; depending on the orbits being considered, a $\nu = 0$ or $\nu = \pi$ section usually is the best choice, so henceforth these sections are referred to as periapse and apoapse sections, respectively. A useful fact is that $\sigma(x,y,p_{x}, p_{y}) = [x + \mu, y] \cdot [p_{x}, p_{y}+ \mu] = 0$ if and only if $\nu = 0$ or $\nu = \pi$. 

As a final note, the osculating orbital elements help define another set of coordinates called \emph{synodic Delaunay variables} $(L, G, \ell, g)$. The latter are important as they are action-angle coordinates \cite{celletti} for the $\mu = 0$ PCRTBP, and thus enable the theory of near-integrable Hamiltonian systems \cite{morbyBook} to be applied to the system. They are defined via $L = \sqrt{(1-\mu)a}$, $G = L\sqrt{1-e^{2}}$, and $\ell = \ell(\nu, e) = $ mean anomaly \cite{bmw}. $g$ remains as defined earlier. Although this paper will use Cartesian coordinates for computations, synodic Delaunay variables can help visualize and interpret results in line with perturbation theory.

\subsection{Unstable Resonant Orbits} \label{mmrsection}

Resonant motions are ubiquitous in celestial systems.
Among the most important resonant phenomena for space missions is that of mean motion resonance (MMR). Roughly speaking\footnote{More rigorously, MMRs are defined through studying the dynamics of  combinations $m g + n \ell $, $m,n \in \mathbb{Z}$ of the synodic Delaunay variables $\ell$ and $g$ from Section \ref{delaunaySection}. See for instance the book \cite{morbyBook} of Morbidelli for more details.}
, an $m$:$n$ MMR is a region of a celestial system's phase space where one body (in our case the spacecraft) makes approximately $m$ revolutions around some large central body (in this case $m_{1}$) in the same time that another body (here, $m_{2}$) makes $n$ revolutions around the same central body. In the spacecraft-moon MMR case, since this is a relation between the spacecraft's orbital period and the fixed period of the moon $m_2$, Kepler's third law \cite{bmw} means that different MMRs correspond to specific spacecraft semimajor axis values, one for each MMR. 

In multi-body celestial systems, MMR regions contain both stable and unstable resonant orbits with various topological properties. Of these, the unstable resonant orbits are of special interest for low-cost orbit transfers; if the unstable manifold of an orbit contained in one MMR intersects the stable manifold of an orbit contained in another MMR, then one gets a zero-fuel heteroclinic trajectory from the first MMR to the second. This \emph{MMR overlap} in turn yields a natural change of spacecraft semimajor axis. Chirikov's overlap criterion \cite{chirikov1960, chirikov1979} states that overlap of MMRs is the key driver of global transport across phase space in celestial systems, and determines the values of semimajor axis a spacecraft can reach without using fuel. 

In the PCRTBP, unstable resonant orbits occur in 1-parameter families of unstable periodic orbits, with one orbit for each Jacobi constant $C$ across some range of $C$ values \cite{kumar2021journal}. For the $m$:$n$ MMR, the resonant periodic orbits will have periods of \emph{approximately} but not exactly $2 \pi n$ (recall that the orbital period of $m_1$ and $m_2$ is $2\pi$ in the normalized units of the PCRTBP). An $m$:$n$ resonant orbit is called an \emph{interior} resonant orbit if $m>n$, or an \emph{exterior} resonant orbit if $m<n$. To study interior resonant orbits, a $\nu = 0$ Poincar\'e section works best, whereas for exterior orbits a $\nu = \pi$ section is better. This avoids having the periodic orbits intersect the section during their closest approaches to $m_{2}$, which is the area where $\dot \nu$ can occasionally briefly become negative and affect transversality to the section. With this choice of Poincar\'e section, an $m$:$n$ resonant periodic orbit will pass through the section $m$ times in one period.

\section{Problem setting, summary, and solution overview}

\subsection{Setting and notation} \label{settingSection}
Consider a 2 DOF Hamiltonian function $H(x, y, p_{x}, p_{y}): M \rightarrow \mathbb{R}$ defined on the phase space $M = \mathbb{R}^{4}$. Assume that the symplectic form is given by the standard symplectic matrix $J = \begin{bmatrix}
0_{2 \times 2}   & I_{2 \times 2}   \\ -I_{2 \times 2}  &  0_{2 \times 2} \end{bmatrix} $ in the usual Euclidean metric on $\mathbb{R}^{4}$. Then, the equations of motion will be 
\begin{equation} \label{H_EOM} \dot x = \frac{\partial H}{\partial p_{x}} \quad\quad \dot y = \frac{\partial H}{\partial p_{y}} \quad \quad \dot p_{x} = -\frac{\partial H}{\partial x} \quad\quad \dot p_{y} = -\frac{\partial H}{\partial y} \end{equation}
Let $\Phi_{t}(\bold{x}) : \mathbb{R}^{4} \rightarrow \mathbb{R}^{4}$ be the time-$t$ flow map of the Hamiltonian dynamical system given by Equation \eqref{H_EOM}, which propagates a point $\bold{x} \in \mathbb{R}^{4}$ by the equations of motion for time $t$. Denote $\mathcal M_C = \{ (x,y,p_{x},p_{y}) : H(x,y,p_{x},p_{y}) = C\}$ as the fixed-energy 3D submanifold of $\mathbb{R}^{4}$ for any given energy level $C$; $\mathcal M_{C}$ will be invariant under (and thus tangent to) the flow. Assume also that we choose some Poincar\'e surface of section $\Sigma$ for the flow of $H$, transverse to the flow in the region of interest. $\Sigma$ will be 3D, and its intersection $\Sigma_C = \Sigma \cap \mathcal M_C$ with $\mathcal M_{C}$ for any $C$ will be 2D in that region. Denote $P: \Sigma \rightarrow \Sigma$ to be the flow's Poincar\'e return map on $\Sigma$, and its restriction to $P$-invariant $\Sigma_C$ as $P_C = \Sigma_C \rightarrow \Sigma_C$. We assume that $\Sigma$ can be represented as the set of $\bold{x} \in \mathbb{R}^{4}$ satisfying  \emph{crossing conditions} $\sigma(\bold{x})=0$ and $\beta(\bold{x})>0$ for some continuous functions $\sigma, \beta: \mathbb{R}^{4} \rightarrow \mathbb{R}$.

Now, consider an unstable periodic orbit $\gamma = \gamma(t)$ of the flow of $H$, having initial condition $\gamma(0) = \bold{x}_{0}$ and period $T$. Hence, $\Phi_{T}(\bold{x}_{0}) = \bold{x}_{0}$, and the orbit's monodromy matrix $D_{\bold{x}} \Phi_{T}(\bold{x}_{0})$ has real eigenvalues $\lambda_{u}, \lambda_{s}$ with $|\lambda_{u}| > 1$  and $\lambda_{s} = \lambda_{u}^{-1}$, in addition to a double eigenvalue of 1 (since the flow direction is always a unit eigenvector of the monodromy matrix, and eigenvalues of symplectic matrices occur in reciprocal pairs). The unstable periodic orbit will thus have 2D stable and unstable manifolds $W^{s}(\gamma)$ and $W^{u}(\gamma)$, which along with $\gamma$ itself must be entirely contained within $\mathcal M_{C}$ with $C = H(\bold{x}_{0})$. Now, assume that $\gamma$ intersects $\Sigma$. Due to the transversality of $\Sigma$ to the flow, over the course of a single period, the orbit $\gamma$ will intersect $\Sigma$ at some finite number $m \in \mathbb{N}$ of discrete points $X(k)$, $k = 0, \dots, m-1$. For the same reason, the 2D manifolds $W^{s}(\gamma)$ and $W^{u}(\gamma)$ will intersect the 2D fixed-energy section $\Sigma_C$, and thus also the 3D section $\Sigma$, along 1D curves. This last fact can be seen by dimension counting; both $W^{u}(\gamma)$ and $\Sigma_C$ are 2D manifolds contained wholly within $\mathcal M_C$, which is 3D. Hence, adding codimensions, the intersection $W^{u}(\gamma) \cap \Sigma_C$ must have codimension 2 and thus dimension 1 inside $\mathcal M_C$, and hence also inside $\Sigma$. The same will similarly hold for $W^{s}(\gamma) \cap \Sigma_C$ as well.

We know that the intersections $W^{u}(\gamma) \cap \Sigma$ and $W^{s}(\gamma) \cap \Sigma$ are 1D submanifolds of $\Sigma$. However, we also know that $\gamma$ intersects $\Sigma$ at $m$ points $X(k)$, each of which will also be contained in $W^{u}(\gamma) \ni \gamma$. Hence, each point $X(k)$ must belong to $W^{u}(\gamma) \cap \Sigma$ which is a 1D manifold; thus, there must be a curve $W^{u}_{p}(k,s): \{0, \dots, m-1\} \times \mathbb{R} \rightarrow \mathbb{R}^{4}$ emanating from each point $X(k)$ such that $W^{u}_{p}(k, 0) = X(k)$ and $Image(W^{u}_{p}) = W^{u}(\gamma) \cap \Sigma$. Similarly for $W^{s}$ there must exist $W^{s}_{p}(k,s)$ such that $Image(W^{s}_{p}) = W^{s}(\gamma) \cap \Sigma$. Henceforth, denote $W^{s}_\Sigma = W^{s}(\gamma) \cap \Sigma$ and $W^{u}_{\Sigma} = W^{u}(\gamma) \cap \Sigma$. Note also that one can interpret $X(k)$ as being a periodic orbit of the Poincar\'e map $P$, with $W^{s}_{p}(k,s)$, $W^{u}_{p}(k,s)$ representing its stable/unstable manifolds.

\begin{remark} 
For the PCRTBP case, $H$ is given by Equation \eqref{pcrtbpH}. For analysis of PCRTBP unstable resonant orbits, as discussed in Section \ref{mmrsection}, $\Sigma$ will be chosen as $\nu=\nu(x,y,p_{x},p_{y}) = \pi$ when studying exterior resonances and $\nu= 0$ when studying interior resonances. Either choice of $\Sigma$ will satisfy $\sigma(x,y,p_{x}, p_{y}) = [x + \mu, y] \cdot [p_{x}, p_{y}+ \mu] = 0$, which occurs if and only if $\nu= 0$ or $\pi$ (recall Section \ref{delaunaySection}); the needed additional condition to distinguish between $\nu= 0$ and $\nu=\pi$ is introduced in Section \ref{sectionSection}. For an unstable $m$:$n$ MMR periodic orbit $\gamma$, it will intersect $\Sigma$ $m$ times, which will thus be the number of points $X(k)$ defined earlier. 
\end{remark}

\subsubsection{Computing the Poincar\'e map $P$ and trajectory intersections with $\Sigma$}\label{sectionSection}

We assumed earlier that the oriented Poincar\'e section $\Sigma$ is $\{\bold{x} \in \mathbb{R}^{4}: \sigma(\bold{x})=0, \beta(\bold{x})>0\}$ for some $\sigma, \beta: \mathbb{R}^{4} \rightarrow \mathbb{R}$. To compute its Poincar\'e map, e.g. in MATLAB or Julia, a continuous $\sigma$ is needed to detect and calculate crossings of trajectories $\bold{x}(t)$ with $\Sigma$ during numerical integration. However, it can occur that not \emph{all} $\bold{x}$ with $\sigma(\bold{x}) = 0$ lie on $\Sigma$; as we discuss later, this is the case for the PCRTBP periapse \& apoapse sections.
 Thus, the condition $\beta(\bold{x})>0$ helps distinguish true crossings of $\Sigma$ from false positives where $\sigma(\bold{x}(t)) = 0$ but $\bold{x}(t) \notin \Sigma$. Often, one can take $\beta(\bold{x})=\dot \sigma(\bold{x})$ or $-\dot \sigma(\bold{x})$, where $\dot \sigma(\bold{x})= \nabla \sigma(\bold{x}) \cdot [\frac{\partial H}{\partial p_{x}}, \, \frac{\partial H}{\partial p_{y}}, \, -\frac{\partial H}{\partial x}, \, -\frac{\partial H}{\partial y} ]$; this simply means that $\sigma$ is increasing (or decreasing) whenever $\bold{x}(t)$ crosses $\Sigma$.

For example, the PCRTBP $\nu = 0$ (periapse) and $\nu= \pi$ (apoapse) sections from Section \ref{delaunaySection} both are zero level sets of $\sigma(x,y,p_{x}, p_{y}) = [x + \mu, y] \cdot [p_{x}, p_{y}+ \mu] = 0$. However, except in a small region near $m_{2}$, $\nu = 0$ only if $\sigma$ crosses zero while increasing, whereas a decreasing $\sigma$ zero-crossing corresponds to $\nu = \pi$. Thus, the condition $\beta>0$ is a fundamental part of distinguishing periapse from apoapse section crossings. 

\begin{remark} 
For computation of trajectory intersections with $\Sigma$, one can use the event detection tools of most modern ODE integration libraries. In MATLAB this is called an ODE {\tt Event}, while in Julia this capability is called a {\tt ContinuousCallback}. One specifies for the numerical integrator to register an event and save the system state every time 1) $\sigma = 0$ AND 2) given a crossing direction, $\dot \sigma > 0$ or $ < 0$.  Both MATLAB and Julia require $\sigma$ as a program input, and can also handle an optional crossing direction input. 

In fact, both MATLAB and Julia can integrate further tests than just $\sigma = 0 $ and $\dot \sigma$ positive/negative into their event detection capabilities. For instance, for the PCRTBP propagated using Cartesian equations \eqref{pcrtbpH_EOM}-\eqref{pcrtbpH} and a $\nu = 0$ section, the most accurate $\Sigma$-intersection test is to not only find when $\sigma = 0$, but to also then directly compute $\nu \in \mathbb{T}$ and verify that the numerical output is near either 0 or $2\pi \approx 6.2832$. One can also similarly include tests checking whether or not $\beta(\bold{x})>0$, even for choices of $\beta(\bold{x})$ other than $\pm \dot \sigma(\bold{x})$. 
\end{remark}

\subsection{Overview of problem and approach} \label{overviewSection}

Many studies of 2 DOF Hamiltonian systems seek to find and characterize heteroclinic connections between periodic orbits. For this, it is necessary to compute the orbits' stable/unstable manifolds as well as their intersections. While unstable periodic orbits have 2D stable/unstable manifolds under the continuous-time flow, computing 2D manifolds and their intersections in the 3D energy submanifold $\mathcal M_C$ requires significantly more computational tools and power than intersecting 1D curves in 2D space. Thus, studies of heteroclinic trajectories in 2 DOF systems generally use a Poincar\'e section $\Sigma$ and its return map $P$ to reduce the flow to Poincar\'e maps $P_C$ on 2D fixed-energy surfaces of section $\Sigma_C = \Sigma \cap \mathcal M_C$. As discussed in the previous section, for a periodic orbit $\gamma$ with 2D stable/unstable manifolds $W^{s}(\gamma)$, $W^{u}(\gamma)$, their intersections $W^{s}_\Sigma$, $W^{u}_\Sigma$ with the 2D surface $\Sigma_C$ will be 1D curves, making detection of heteroclinics much easier. 

The goals of this paper are to 1) present accurate and computationally efficient methods for computing $W^{s}_\Sigma$ and $W^{u}_\Sigma$, and 2) to use the resulting manifolds to accurately compute heteroclinic connections. More specifically, given $H$, $\Sigma$, and $\gamma$ as in Section \ref{settingSection},  goal 1 seeks to accurately compute the functions $W_{p}^{u}(k,s)$ and $W_{p}^{s}(k,s)$ representing $W^{s}_\Sigma$, $W^{u}_\Sigma$, the stable/unstable manifolds' intersection curves with $\Sigma$. Then, given such $W_{1p}^{u}(k_{1},s_{1})$ and $W_{2p}^{s}(k_{2},s_{2})$ representing the unstable and stable manifold $\Sigma$-intersection curves $W^{u}_{1\Sigma}=W^{u}(\gamma_{1}) \cap \Sigma$ and $W^{s}_{2\Sigma}=W^{s}(\gamma_{2}) \cap \Sigma$ of periodic orbits $\gamma_{1}$ and $\gamma_{2}$, goal 2 will seek to find $(k_{1},s_{1}, k_{2},s_{2})$ such that $W_{1p}^{u}(k_{1},s_{1}) = W_{2p}^{s}(k_{2},s_{2})$. This yields a heteroclinic intersection between the manifolds. 

\medskip
Goal 1 is discussed in Section \ref{parambigsection}; it was solved in the case of $m = 1$ in our previous work \cite{kumar2021journal}, where $X(k)$ becomes a fixed point $X(0)$ of $P$ rather than a periodic orbit. For $m>1$, a similar problem with maps defined by closed-form expression rather than Poincar\'e maps was addressed in \cite{gonzalezJames}, but requiring composition of the map with Taylor series, which is a fairly involved process for Poincar\'e maps. We wish to find the manifold curves $W^{s}_\Sigma$, $W^{u}_\Sigma$ without this composition.
Thus, to obtain accurate representations $W_{p}^{u}$ and $W_{p}^{s}$ of periodic orbit stable/unstable manifolds as described in Section \ref{settingSection}, we will follow these basic steps:
\begin{enumerate} 
\item  Given $\gamma$, find its intersection points $X(k)$ with $\Sigma$. Then gather some information about the linearized dynamics near the periodic orbit at each $X(k)$. These include stable and unstable multipliers and eigenvectors as well as tangent directions to the orbit, which are used to construct an adapted frame . 
 
\item Find functions of form $W(k, s): \{0, \dots, m-1\} \times \mathbb{R} \rightarrow \mathbb{R}^{4}$, each of which  yields $m$ curves that lie on $W^{u}(\gamma)$ (or $W^{s}(\gamma)$ ) and \emph{near but not on} $\Sigma$, at least for small $s$. Using a parameterization method with the adapted frame from step 1, solve for $W$ as Taylor series valid in some interval around $s = 0$. 
 
\item Given the parameterizations of curves from step 2, determine their domain of validity around $s = 0$. 

\item Evaluate the parameterizations $W$ on a fine grid of valid $s$ values, and propagate the resulting points to $\Sigma$. This also determines the values of the functions $W^{s}_{p}$ and $W_{p}^{u}$ representing $W^{s}_\Sigma$, $W^{u}_\Sigma$, at least for $s$ within the previously-determined radius of validity (which we refer to as the \emph{local manifolds}). 

\item Compute $W^{s}_{p}$ and $W_{p}^{u}$ for $s$ outside the Taylor parameterizations' domains of validity by applying $P$ or $P^{-1}$ to points from the local manifolds (globalization).
\end{enumerate}
In step 2, by relaxing the requirement for the curves parameterized by $W$ to lie exactly on $\Sigma$, one avoids having to compose Poincar\'e maps with Taylor series, as desired. The calculations carried out in steps 4 and 5 above also yield many points on the curves $W^{s}_\Sigma$, $W^{u}_\Sigma$, which can be plotted on the 2D fixed-energy Poincar\'e section $\Sigma_{C}$ for further analysis. We fully explain the above steps in the following Section \ref{parambigsection}. 
\medskip

Goal 2 is detailed in Section \ref{heteroclinicSection}, but we give a brief preview of its solution steps now as well. Assume that we have found the functions $W_{1p}^{u}(k_{1},s_{1})$ and $W_{2p}^{s}(k_{2},s_{2})$ representing the unstable/stable manifold curves $W^{u}_{1\Sigma}=W^{u}(\gamma_{1}) \cap \Sigma$ and $W^{s}_{2\Sigma}=W^{s}(\gamma_{2}) \cap \Sigma$ of orbits $\gamma_{1}$ and $\gamma_{2}$, along with the coordinates of many points on those curves. Then, to solve the heteroclinic intersection equation $W_{1p}^{u}(k_{1},s_{1}) = W_{2p}^{s}(k_{2},s_{2})$, we will:
\begin{enumerate} 
 
\item Define subsets $U_{N} \subset W^{u}_{1\Sigma}$ and $S_{N} \subset W^{s}_{2\Sigma}$ for $N \in \mathbb{Z}$ such that one can restrict the heteroclinic solution search to $(k_{1},s_{1}, k_{2},s_{2})$ belonging only to subset pairs of form $(U_{N}, S_{N})$ and $(U_{N}, S_{N-1})$. 
 
\item For each subset pair $(U_N, S_M)$, take the known points from the curves $W^{u}_{1\Sigma}$, $W^{s}_{2\Sigma}$ that belong to those subsets, and consider the line segments connecting consecutive points on each curve. Check whether the 2D projection of any line segment from $U_{N}$ intersects any projected segment from $S_{M}$, and if so, store the corresponding $(k_{1},k_{2})$ as well as segment endpoint $s_{1}$, $s_{2}$ values for each intersection point. 

\item For each intersection of line segments found in step 2, use the resulting $(k_{1}, k_{2})$ and bounds on $(s_{1}, s_{2})$ to refine $s_{1}$ and $s_{2}$ by bisection, until the error $\|W_{1p}^{u}(k_{1},s_{1}) - W_{2p}^{s}(k_{2},s_{2})\|$ is within (small) tolerance. 
\end{enumerate}
Assuming that the manifold curve points used to define the line segments of step 2 are spaced closely enough together, the resulting values and bounds for $(k_{1},s_{1}, k_{2},s_{2})$ should be sufficiently accurate for step 3 to quickly yield a highly accurate refined solution. The above steps are detailed in Section \ref{heteroclinicSection}.

\section{Parameterization and computation of invariant manifolds} \label{parambigsection}

Assume that one is given a 2 DOF Hamiltonian $H$ with time-$t$ flow map denoted $\Phi_{t}(\bold{x})$, a periodic orbit $\gamma$ with known initial condition $\gamma(0) = \bold{x}_{0}$ and period $T$, and a Poincar\'e surface of section $\Sigma = \{\bold{x} \in \mathbb{R}^{4}: \sigma(\bold{x})=0, \beta(\bold{x})>0\}$ for some continuous $\sigma, \beta:\mathbb{R}^{4} \rightarrow \mathbb{R}$. In this section, the 5 steps summarized in goal 1 of Section \ref{overviewSection} are explained in detail, describing the methods to accurately compute functions $W^{u}_{p}, W^{s}_{p}: \{0, \dots, m-1\} \times \mathbb{R} \rightarrow \mathbb{R}^{4}$ representing the manifold curves $W^{s}_\Sigma$ and $W^{u}_\Sigma$ defined in Section \ref{settingSection}.

\subsection{Step 1: points $X(k)$ on $\Sigma$ and adapted frame} \label{step1}

The first part of step 1 is to find the $m$ intersection points $X(k)$, indexed by $k = 0, \dots, m-1$, of the orbit $\gamma$ with the section $\Sigma$. This can be done by numerically integrating $\bold{x}_{0}$ by its known period $T$ and using the event detection tools described in Section \ref{sectionSection} to calculate each crossing of $\gamma(t)$ with $\Sigma$ over $t \in [0, T)$; the first $\Sigma$ crossing of $\gamma(t)$ should be saved as $X(0)$, the second crossing as $X(1)$, and so on. Furthermore, as part of this step, one should detect the times $t_{k} \in [0, T)$ from this integration when $\gamma(t_{k}) = X(k)$. Then, define $\tau(k) = t_{k+1} - t_{k}$ for $k \in \{0, \dots, m-2\}$ and $\tau(m-1) = (t_{0}+T) -t_{m-1}$. Each $\tau(k)$ so defined is the Poincar\'e map first return time of the point $X(k)$: the time it takes for $X(k) \in \Sigma$ to return to $\Sigma$ when propagated by the flow. Since we know that the orbit $\gamma(t)$ goes through the $X(k)$ in order, we have that 
\begin{equation} \label{poInvariance} \Phi_{\tau(k)}(X(k)) = P(X(k)) = X(k+1 \mod m) \end{equation}
for all $k = 0, \dots, m-1$. Since all $X(k)$ lie on the section $\Sigma$ and $\tau(k)$ are their first return times, $P(X(k)) = \Phi_{\tau(k)}(X(k))$ here. However, note that in general, the Poincar\'e map $P:\Sigma \rightarrow \Sigma$ is \emph{not} the same as the map $\Phi_{\tau(k)}: \mathbb{R}^{4} \rightarrow \mathbb{R}^{4}$ for any $k$, as the latter is a fixed-time flow map on all of $\mathbb{R}^{4}$. Also note that Equation \eqref{poInvariance} is just the usual multiple shooting equation for periodic orbits of a map $P$. 

Now, given the points $X(k) \in \Sigma$, $k = 0, \dots, m-1$ of $\gamma$ lying on $\Sigma$ and satisfying Equation \ref{poInvariance}, one needs to compute some properties of the linearized flow. First of all, compute monodromy matrices $D_{\bold{x}}\Phi_T(X(k))$ of $\gamma$ at each point $X(k)$. Finding their eigenvalues and eigenvectors in turn will yield the stable \& unstable Floquet multipliers of the orbit $\tilde \lambda_{s}$ and $\tilde \lambda_{u}$, respectively, as well as stable \& unstable unit-length eigenvectors $\bold{v}_{s}(k)$ and $\bold{v}_{u}(k)$ at each point $X(k)$. These eigenvectors are tangent to the stable/unstable manifolds at each point, and thus linearly approximate the local manifolds. 

However, there is ambiguity in the orientation of each $\bold{v}_{s}$ and $\bold{v}_{u}$ here; if $\bold{v}_{s}$ is a unit-length stable eigenvector then so is $-\bold{v}_{s}$ (and similarly with $\bold{v}_{u}$). To resolve this, first recall the standard result that monodromy matrices at different points $X(k)$ of the orbit are similar, as $D_{\bold{x}}\Phi_{T}(X(k+1 \mod m)) = D_{\bold{x}}\Phi_{\tau(k)}(X(k)) D_{\bold{x}}\Phi_{T}(X(k)) D_{\bold{x}}\Phi_{\tau(k)}(X(k))^{-1}$; this in turn implies that if $\bold{v}$ is an eigenvector of $D_{\bold{x}}\Phi_{T}(X(k)) $ with eigenvalue $\tilde \lambda$, then $D_{\bold{x}}\Phi_{\tau(k)}(X(k)) \bold{v}$ is an eigenvector of $D_{\bold{x}}\Phi_{T}(X(k + 1 \mod m))$ for $\tilde \lambda$ as well. Since $\tilde \lambda_{s}$ and $\tilde \lambda_{u}$ each have 1D eigenspaces spanned by $\bold{v}_{s}(k)$ and $\bold{v}_{u}(k)$ for each monodromy matrix $D_{\bold{x}}\Phi_T(X(k))$, this implies that there exist $\lambda_{s}(k), \lambda_{u}(k)$ such that $D_{\bold{x}}\Phi_{\tau(k)}(X(k)) \bold{v}_{s}(k) = \lambda_{s}(k) \bold{v}_{s}(k + 1 \mod m)$ (similar for $\lambda_{u}(k)$ and $\bold{v}_{u}(k)$). Now, to resolve the ambiguity in orienting $\bold{v}_{s}$ and $\bold{v}_{u}$, first arbitrarily choose directions for $\bold{v}_{s}(0)$ and $\bold{v}_{u}(0)$, and then recursively choose orientations for the unit eigenvectors $\bold{v}_{s}(i)$ and $\bold{v}_{u}(i)$, $i = 1, \dots, m-1$ such that for $k = 0, \dots, m-2$, 
\begin{equation} \label{condition}
\begin{split}
\lambda_{s}(k) = {[\bold{v}_{s}(k+1 \mod m)]^{T} D_{\bold{x}}\Phi_{\tau(k)}(X(k)) \bold{v}_{s}(k)} > 0 \\ 
\lambda_{u}(k) = {[\bold{v}_{u}(k+1 \mod m)]^{T} D_{\bold{x}}\Phi_{\tau(k)}(X(k)) \bold{v}_{u}(k)} > 0
\end{split}
\end{equation}

In later steps (e.g. Equations \eqref{cohom_as}-\eqref{cohom_au}), we require Equation \eqref{condition} and $\lambda_{u}(k), \lambda_{s}(k) > 0$ to be satisfied not only for $k = 0, \dots, m-2$, but also for $k = m-1$. However, if the periodic orbit being considered has monodromy matrix stable/unstable eigenvalues $\tilde \lambda_{s}$, $\tilde \lambda_{u}$ negative, then not all of the $\lambda_{u}(k), \lambda_{s}(k)$ can be positive; to handle such cases, one can consider double covers of such periodic orbits. That is, rather than considering periodic orbit $\gamma(t)$ with period $T$, study the orbit $\gamma_{d}(t) = \gamma(t \mod T)$ over time $t \in [0, 2T]$, which should be considered as a $2T$ periodic orbit. The monodromy matrices of $\gamma_d$ will have positive eigenvalues, as they will be squares of those of $\gamma$, while the stable/unstable manifolds of $\gamma_d$ will match those of $\gamma$. So, in such cases, one should study the double cover orbit rather than the original.

With unit eigenvectors $\bold{v}_{s}(k)$ and $\bold{v}_{u}(k)$ now set satisfying Equation \eqref{condition} for $k = 0, \dots, m-1$, one should now rescale them to make the previously-defined multipliers $\lambda_{s}(k)$, $\lambda_{u}(k)$ constant. This will make the manifold computations in Section \ref{paramsection} significantly easier. To find the desired rescaling, solve for $a_{s}(k)$, $a_{u}(k)$ for $k = 0, \dots, m-1$ and positive constants $\bar \lambda_{s}$, $\bar \lambda_{u}$ satisfying the equations
\begin{equation} \label{cohom_as} \log(a_{s}(k)) - \log(a_{s}(k+1 \mod m)) = -[\log (\lambda_{s}(k)) - \log \bar \lambda_{s}] \end{equation}
\begin{equation} \label{cohom_au} \log(a_{u}(k)) - \log(a_{u}(k+1 \mod m)) = -[\log (\lambda_{u}(k)) - \log \bar \lambda_{u}] \end{equation}
Note that the sum over $k = 0, \dots, m-1$ of the LHS of each equation is zero, so that the same should be true of the RHS. This yields that $\bar \lambda_{s} = \exp\left(\frac{1}{m}\sum_{k=0}^{m-1} \log \lambda_{s}(k) \right) $ and $\bar \lambda_{u} = \exp\left(\frac{1}{m}\sum_{k=0}^{m-1} \log \lambda_{u}(k) \right) $. The solution of $a_s(k)$ and $a_u(k)$ follows a procedure which will be needed many times in this paper, and is described at the end of this Step 1 subsection, in Section \ref{cohomSolving}. For now, assume that $a_{s}(k)$, $a_{u}(k)$, $\bar \lambda_{s}$, and $\bar \lambda_{u}$ have been found as described earlier. Then, defining $\bold{\bar v}_{s}(k) = a_{s}(k) \bold{v}_{s}(k)$, we will have that $D_{\bold{x}}\Phi_{\tau(k)}(X(k))  \bold{\bar v}_{s}(k) =  \bar \lambda_{s} \bold{\bar v}_{s}(k + 1 \mod m)$, and similarly for $\bold{\bar v}_{u}(k) = a_{u}(k) \bold{\bar v}_{u}(k)$; this is proven in \ref{rescaleProof}. 

The vectors $\bold{\bar v}_{s}(k)$, $\bold{\bar v}_{u}(k)$ for $k=0, \dots, m-1$ found above are two of the four vectors needed at each point $X(k)$ to form an \emph{adapted frame} in which later computations will be much simpler. For the other two vectors of this frame, denoted $\bold{\bar v}_{1}(k)$ and $\bold{\bar v}_{2}(k)$, the $\bold{\bar v}_{1}(k)$ should simply be taken as the flow vector $ [\frac{\partial H}{\partial p_{x}}, \, \frac{\partial H}{\partial p_{y}}, \, -\frac{\partial H}{\partial x}, \, -\frac{\partial H}{\partial y} ]$ at each $X(k)$. Given Equation \eqref{poInvariance} and well known properties of flows, this implies that $D_{\bold{x}}\Phi_{\tau(k)}(X(k))  \bold{\bar v}_{1}(k) =  \bold{\bar v}_{1}(k + 1 \mod m)$. To find $\bold{\bar v}_{2}(k)$, the following steps are required, as adapted from a similar procedure developed in \cite{kumar2022} for invariant tori:

\begin{enumerate} 
\item Find $T(k), B(k), C(k), D(k) \in \mathbb{R}$ for $k = 0, \dots, m-1$ satisfying
\begin{equation}  \label{abcd}
\begin{aligned} D\Phi_{\tau(k)} \frac{J^{-1} \bold{\bar v}_{1}(k)}{ \|\bold{\bar v}_{1}(k)\|^{2}} = T(k) &\bold{\bar v}_{1}(k+1 \mod m) + B(k) \frac{J^{-1} \bold{\bar v}_{1}(k+1 \mod m)}{ \|\bold{\bar v}_{1}(k+1 \mod m)\|^{2}} \\ &+ C(k) \bold{\bar v}_{s}(k+1 \mod m) + D(k) \bold{\bar v}_{u}(k+1 \mod m) \end{aligned} \end{equation}
For each $k$ this is just a $4 \times 4$ linear system of equations that can be solved for $T(k), B(k), C(k)$, and $D(k)$. One should find that $B(k) = 1$ for all $k$, as is proven in \ref{sympConjProof}. 

\item Using the $C(k)$, $D(k)$ found earlier, solve for $f_{1}(k)$ and $f_{2}(k)$ for $k = 0, \dots, m-1$ satisfying 
\begin{equation} \label{cEquation} C(k) =  f_{1}(k+1 \mod m) - \bar \lambda_{s}f_{1}(k) \end{equation}
\begin{equation}  \label{dEquation}  D(k) =  f_{2}(k+1 \mod m) - \bar \lambda_{u}f_{2}(k)  \end{equation}
A method of efficiently solving such equations is given in Section \ref{cohomHypSolving}.

\item Using the $f_{1}(k)$ and $f_{2}(k)$ found just earlier, compute  $\bold{v}_{2}(k)$ as 
\begin{equation} \label{prelimSympConj} \bold{ v}_{2}(k) = \frac{J^{-1} \bold{\bar v}_1(k)}{ \|\bold{\bar v}_1(k)\|^{2}} + f_{1}(k) \bold{\bar v}_{s}(k)  + f_{2}(k) \bold{\bar v}_{u}(k)  \end{equation}

\item Define $\bar T = \frac{1}{m}\sum_{k=0}^{m-1} T(k) $ and solve (see Section \ref{cohomSolving}) for $a(k)$, $k = 0, \dots, m-1$ satisfying
\begin{equation}  \label{Tkill}  a(k) - a(k+1 \mod m) = -\left(T(k)-\bar T\right) \end{equation}

\item Set each $\bold{\bar v}_{2}(k) = \bold{v}_{2}(k) + a(k) \bold{\bar v}_{1}(k) $

\end{enumerate} 
As is proven in \ref{sympConjProof}, the vectors $\bold{\bar v}_{2}(k)$ found using this procedure satisfy the equation 
\begin{equation}  \label{sympConjEquation} D_{\bold{x}}\Phi_{\tau(k)}(X(k))  \bold{\bar v}_{2}(k) =  \bar T \bold{\bar v}_{1}(k + 1 \mod m)  + \bold{\bar v}_{2}(k + 1 \mod m)   \end{equation}

Now, define matrices $M(k) \in \mathbb{R}^{4 \times 4}$ with column 1 given by $\bold{v}_{1}(k)$, column 2 by $\bold{\bar v}_{2}(k)$, column 3 by $\bold{\bar v}_{s}(k)$, and column 4 by $\bold{\bar v}_{u}(k)$ for $k=0, 1, \dots, m-1$. Then, recalling Equation \eqref{sympConjEquation} along with the previous results on products of $D_{\bold{x}}\Phi_{\tau(k)}(X(k))$ with $ \bold{\bar v}_{1}(k)$, $  \bold{\bar v}_{s}(k)$, and $ \bold{\bar v}_{u}(k)$ yields that for all $k=0, 1, \dots, m-1$, 
\begin{equation}  \label{frameEquation} D_{\bold{x}}\Phi_{\tau(k)}(X(k))  M(k) =  M(k + 1 \mod m)  \Lambda \quad \quad \text{where } \Lambda = \begin{bmatrix}
1 &  \bar T   & 0 & 0 \\ 0 &  1   & 0 & 0 \\ 0 & 0  & \bar \lambda_s & 0 \\ 0 &  0 & 0 & \bar \lambda_u \end{bmatrix}  \end{equation}
The matrices $M(k)$ are referred to as an \emph{adapted frame}, and the constant matrix $\Lambda$ is a Floquet matrix whose near-diagonal form will simplify the manifold computations to follow. 

\subsubsection{Solving ``cohomological equations'' such as Equations \eqref{cohom_as}-\eqref{cohom_au}, Equation \eqref{Tkill}} \label{cohomSolving}
Equations \eqref{cohom_as}-\eqref{cohom_au} and \eqref{Tkill} are all of a general form which will be seen many times in this paper:
\begin{equation} \label{cohom} u(k) - u(k+1 \mod m) = b(k) \end{equation}
with $b(k)$ some known function of $k = 0, \dots, m-1$ satisfying $\sum_{k=0}^{m-1} b(k) =0$. Note that if $u(k)$ is a solution of Equation \eqref{cohom}, then so is $u(k)+C$ for any $C \in \mathbb{R}$. Thus, to solve for $u(k)$, one can set $u(0) = 0$ arbitrarily, and then recursively find $u(1), \dots, u(m-1)$ using Equation \eqref{cohom} for $k = 0, \dots, m-2$. Equation \eqref{cohom} will then automatically also be satisfied for $k=m-1$ due to the condition $\sum_{k=0}^{m-1} b(k) =0$. Note that Equation \eqref{cohom} is analogous to the cohomological equations $u(\theta) - u(\theta+\omega) = b(\theta)$ involved in computing invariant tori and their adapted frames, see e.g. \cite{kumar2022}. 

\subsubsection{Fixed-point iteration for solving equations such as Equations \eqref{cEquation}-\eqref{dEquation}} \label{cohomHypSolving}
Equations \eqref{cEquation}-\eqref{dEquation} are of a general form which will also be seen many times in this paper:
\begin{equation} \label{cohomHyp} \alpha u(k) -  u(k+1 \mod m) = b(k) \end{equation}
with $1 \neq \alpha >0$ and $b(k)$ some known function of $k = 0, \dots, m-1$.
To find $u(k)$, rewrite Equation \eqref{cohomHyp} as
\begin{equation} \label{u1contract} u(k) = \alpha u(k-1 \mod m) -  b(k-1 \mod m)  \stackrel{\text{def}}{=} [A(u)](k)\end{equation} 
\begin{equation} \label{u2contract} u(k) = \alpha^{-1} \left[ b(k) +   u(k+1 \mod m) \right]  \stackrel{\text{def}}{=}  [B(u)](k) \end{equation}  
$A$ and $B$ are maps which send any finite sequence $u(k)$, $k = 0, \dots, m-1$ to the new finite sequences $A(u)$ and $B(u)$ with $k$th terms given by the middle expressions of Eq. \eqref{u1contract}-\eqref{u2contract} for $k = 0, \dots, m-1$. 
It is then easy to show (see \ref{contractionProof}) that if $\alpha < 1$, $A$ is a contraction under the $\ell^\infty$ norm, and similarly for $B$  if $\alpha > 1$. Thus, to find $u$, let $u_0(k) = 0$ for all $k = 0, \dots, m-1$, and repeatedly iterate $u_{n+1} = A(u_{n})$ (if $\alpha < 1$) or $u_{n+1} = B(u_{n})$ (if $\alpha > 1$), starting at $n=0$. By the contraction mapping theorem  \cite{chicone2006}, the iteration will converge to the desired solution sequence $u$ of Eq. \eqref{u1contract} or  \eqref{u2contract}, and thus also of Eq. \eqref{cohomHyp}. 

\subsection{Step 2: functions $W(k,s)$ parameterizing curves on $W^s(\gamma), W^u(\gamma)$ near $\Sigma$} \label{paramsection}

With intersection points $X(k) \in \Sigma$, adapted frame matrices $M(k) \in \mathbb{R}^{4 \times 4}$, and Floquet matrix $\Lambda$ computed for periodic orbit $\gamma$, we now start the computation of its stable/unstable manifolds $W^s(\gamma), W^u(\gamma)$. Recall from Section \ref{overviewSection} that the goal is to accurately compute functions $W_{p}^{u}, W_{p}^{s}: \{0, 1, \dots, m-1\} \times \mathbb{R} \rightarrow \Sigma$ with $W_{p}^{u}(k,0)=W_{p}^{s}(k,0)=X(k)$ parameterizing $W^{s}_\Sigma$, $W^{u}_\Sigma$, the stable/unstable manifolds' 1D intersection curves with the Poincar\'e section $\Sigma$. The equation that $W_{p}^{u}$ and $W_{p}^{s}$ should satisfy is
\begin{equation} \label{wpInvariance} P(W_{p}(k,s)) = W_{p}(k+1 \mod m, \lambda s) \end{equation}
where $\lambda = \bar \lambda_s$ or $\bar \lambda_u$ depending on which manifold (stable $W_{p}^{s}$ or unstable $W_{p}^{u}$) is being computed. We also require that $W_{p}(k,0)=X(k)$, which for $s=0$ makes Equation \eqref{wpInvariance} equivalent to Equation \eqref{poInvariance}. Solving Equation \eqref{wpInvariance} directly as in \cite{gonzalezJames} however requires composing $P$ with Taylor series, which we wish to avoid. 

To compute the manifolds without solving Equation \eqref{wpInvariance} directly, we will instead first solve the equation
\begin{equation} \label{invariancequationfinal} \Phi_{\tau(k)}(W(k,s)) = W(k+1 \mod m, \lambda s) \end{equation}
which involves involves fixed-time mappings $\Phi_{\tau(k)}$, rather than the Poincar\'e map $P$. Here, $W: \{0, 1, \dots, m-1\} \times \mathbb{R} \rightarrow \mathbb{R}^{4}$ still will be a function parameterizing $m$ 1D curves lying on $W^s(\gamma)$ or $W^u(\gamma)$, with $W(k,0)=X(k)$. However, these curves will no longer necessarily lie \emph{in} $\Sigma$, instead lying \emph{near} $\Sigma$ for $s$ near 0. Equation \eqref{invariancequationfinal} simplifies computing $W$ by allowing composition of fixed-time $\Phi_{\tau(k)}$ instead of $P$ with Taylor series. 

Interpreting Equation \eqref{invariancequationfinal} in the parameterization method framework of Section \ref{paramethodsection}, the model manifold here is $\mathcal K = \{0, 1, \dots, m-1\} \times \mathbb{R}$, and $f(k, s) = (k+1 \mod m, \lambda s)$, where $\lambda$ is $\bar \lambda_s$ or $\bar \lambda_u$ depending on whether the stable or unstable manifold is being computed. In line with other parameterization methods for fixed points \cite{kumar2021journal} and for invariant tori \cite{kumar2022}, we will seek a solution in $m$ Taylor series
\begin{equation}  \label{series} W(k,s) = X(k) + \sum_{d \geq 1} W_{d}(k) s^d  \end{equation}
where $W_{d}(k) \in \mathbb{R}^{4}$ for all $k=0, 1, \dots, m-1$. As desired, $W(k,0)=X(k)$, so $s=0$ corresponds to the periodic orbit whose manifold we are trying to compute. The $s^{0}$ term of $W$ is $X(k)$, and the linear term $W_{1}(k)$ will be the stable $\bold{\bar v}_{s}(k)$ or unstable $\bold{\bar v}_{u}(k)$ defined in Section \ref{step1} and contained in the third or fourth column of $M(k)$. Hence we need to solve for the higher-order coefficients $W_{d}(k) \in \mathbb{R}^{4}$, $d \geq 2$. 

Denote $W_{<d} (k, s) = X(k) + \sum_{j = 1}^{d-1} W_{j}(k)s^{j}  $. Assume we have solved for all $W_{j}(k)$ for $j <d$, so that the Taylor expansions of $\Phi_{\tau(k)}(W_{<d}(k, s)) - W_{<d}( k+1 \mod m, \lambda s)$ have only $s^{d}$ and higher order terms for all $k=0, 1, \dots, m-1$. Then, starting with $d=2$, the recursive method to solve for the $W_{d}(k)$ is:
    \begin{enumerate}
        	\item Find $E_{d}(k)= [\Phi_{\tau(k)}(W_{<d}(k, s)) - W_{<d}(k+1 \mod m,  \lambda s)]_{d}$, where $[\cdot]_{d}$ denotes the $s^{d}$ Taylor coefficient of the term inside brackets. We describe methods for this in Section \ref{jettransport}. 
	\item Find $W_{d} (k)$, $k=0, 1, \dots, m-1$ such that $W_{<d} (k,s)+W_{d} (k)s^{d}$ cancels the error $E_d (k)s^{d}$ in Eq. \eqref{invariancequationfinal}, thus satisfying Eq. \eqref{invariancequationfinal} up to order $s^{d}$. The equation to solve for $W_{d}(k)$ is 
	\begin{equation}  \label{correctionwk} D_{\bold{x}}\Phi_{\tau(k)}(X(k)) W_{d}(k) - \lambda^{d} W_{d}(k+1 \mod m) = -E_{d}(k) \end{equation}
	Making the substitution $W_{d}(k) = M(k) V_d(k)$, defining $\eta_d(k) = -M(k+1 \mod m)^{-1}E_{d}(k) $, and recalling Eq. \eqref{frameEquation}, Equation \eqref{correctionwk} is equivalent to $\Lambda V_{d}(k) - \lambda^{d} V_{d}(k+1 \mod m) = \eta(k)$, or equivalently:
	\begin{equation}  \label{v1}  V_{d,1}(k) - \lambda^{d} V_{d,1}(k+1 \mod m) = \eta_{d,1}(k) - \bar T  V_{d,2}(k) \end{equation} 
		\begin{equation}  \label{v2}  V_{d,2}(k) - \lambda^{d} V_{d,2}(k+1 \mod m) = \eta_{d,2}(k)  \end{equation} 
	\begin{equation}  \label{v3} \bar \lambda_s V_{d,3}(k) - \lambda^{d} V_{d,3}(k+1 \mod m) = \eta_{d,3}(k)   \end{equation} 
	\begin{equation}  \label{v4} \bar \lambda_u V_{d,4}(k) - \lambda^{d} V_{d,4}(k+1 \mod m) = \eta_{d,4}(k)   \end{equation} 
where $V_d(k) = [V_{d,1}(k), V_{d,2}(k), V_{d,3}(k), V_{d,4}(k)]^T , \eta_d(k) = [\eta_{d,1}(k), \eta_{d,2}(k), \eta_{d,3}(k), \eta_{d,4}(k)]^T \in \mathbb{R}^4$. After multiplying through by $\lambda^{-d}$, Equations \eqref{v1}-\eqref{v4} can be solved by the method of Section \ref{cohomHypSolving}, since $\bar \lambda_u \lambda^{-d} \neq 1$, $\bar \lambda_s \lambda^{-d} \neq 1$ always for $d \geq 2$. This gives $V_d(k)$ and thus $W_{d}(k) = M(k) V_d(k)$. 

	\item Set $W_{<d+1} (k,s) = W_{<d} (k, s) + W_d(k) s^{d}$ and return to step 1.
    \end{enumerate}
    The recursion is stopped when we are satisfied with the degree $d$ of $W$. Note that the adapted frame  allowed us to nearly decouple the equations in Step 2 (except for a single back-substitution), simplifying the solution of Equation \eqref{correctionwk}. We now prove that Equation \eqref{correctionwk} indeed yields $W_{d}(k)$ cancelling the order $s^d$ error. 

\begin{claim} If $W_{d}$ solves Eq. \eqref{correctionwk}, then for $j \leq d$ (using the $[\cdot]_{d}$ notation defined earlier),
\begin{equation}  \label{jcoeff} \left[  \Phi_{\tau(k)}(W_{<d}(k, s)+W_{d}(k)s^{k}) - \left(W_{<d}(k+1 \mod m, \lambda s)+W_{k}(k+1 \mod m)( \lambda s)^{d}\right) \right]_{j}= 0 \end{equation} 
\end{claim} 
\begin{proof} Recall that $ \Phi_{\tau(k)}(W_{<d}(k, s)) - W_{<d}(k+1 \mod m,  \lambda s)=E_{d}(k)s^{d} + \mathcal O(s^{d+1}) $ by assumption. Expanding Eq. \eqref{jcoeff} in Taylor series and keeping only $s^{d}$ and lower order terms gives
\begin{align}  \begin{split}
\Big[\Phi_{\tau(k)}(W_{<d}&(k, s)) + \Big. D_{\bold{x}}\Phi_{\tau(k)}(W_{<d}(k, s)) W_{d}(k)s^{d} - \\
& \quad \quad \quad \left. \left(W_{<d}(k+1 \mod m, \lambda s)+W_{d}(k+1 \mod m)( \lambda s)^{d}\right) \right]_{j} \\
=&[E_{d}(k)s^{k} + D_{\bold{x}}\Phi_{\tau(k)}(W_{<d}(k,s)) W_{d}(k)s^{d} - \lambda^{d} W_{d}(k+1 \mod m)s^{d} ]_{j} \\
=&\begin{cases} 
      0 &\text{if $j<d$}, \\
      E_{d}(k) + D_{\bold{x}}\Phi_{\tau(k)}(X(k)) W_{d}(k) - \lambda^{d} W_{d}(k+1 \mod m) = 0 &\text{if $j=d$}
   \end{cases} \\
\end{split} \end{align}
where the $j=d$ case of the last line follows from the preceding line by dividing $s^{d}$ out from the quantity inside $[.]_{j}$, and then taking $s \rightarrow 0$. 
\end{proof}

\subsubsection{Computing $E_{d}(k)$: automatic differentiation and jet transport} \label{jettransport}

In step 1 of the order-by-order parameterization method to find $W$, we computed the $s^{d}$ coefficients
\begin{equation} \label{Ekdef} E_{d}(k)= [\Phi_{\tau(k)}(W_{<d}(k, s)) - W_{<d}(k+1 \mod m,  \lambda s)]_{d}\end{equation}
Each $W_{<d}(k, s)$ is a degree $d-1$ polynomial and $\lambda$ is a constant, so the $s^{d}$ term of $W_{<d}(k+1 \mod m,  \lambda s)$ is just 0. 
However, computing the Taylor expansion of $\Phi_{\tau(k)}(W_{<d}(k, s)) $ is more complicated, as each $\Phi_{\tau(k)}$ is a nonlinear map defined by integrating points for a fixed time $\tau(k)$ by the Hamiltonian equations of motion \eqref{H_EOM}. For this, we use the tools of automatic differentiation \cite{haroetal} and jet transport \cite{perezpalau2015}, which are also sometimes referred to as differential algebra in the literature \citep{dast,Berz1998}. The following discussion of these two methods is largely identically reproduced from the author's previous paper \cite{kumar2022}. 

Automatic differentiation is an efficient and recursive technique for evaluating operations on Taylor series. For instance, let $f(s)$ and $g(s)$, $s \in \mathbb{R}$, be two series; we can use their known coefficients to compute $d(s) = f(s) / g(s)$ as a Taylor series as well. Let subscript $j$ denote the $s^{j}$ coefficient of a series; since $ f(s) = d(s) g(s)$, we find that $f_{i} = \sum_{j=0}^{i} d_{j} g_{i-j} =  \left( \sum_{j=0}^{i-1} d_{j} g_{i-j}(s) \right)+ d_{i} g_{0} $, which implies that
\begin{equation}  
  \label{autodiff}  d_{i} = \frac{1}{g_{0}} \left( f_{i} - \sum_{j=0}^{i-1} d_{j} g_{i-j} \right)\end{equation} 
Starting with $d_{0} = f_{0}/ g_{0}$, Eq. \eqref{autodiff} allows us to recursively compute $d_{i}$, $i \geq 1$. Similar formulas also exist for recursively evaluating many other functions and operations on Taylor series, including $f(s)^{\alpha}$, $\alpha \in \mathbb{R}$; see \cite{haroetal} for more examples. Most importantly, in all automatic differentiation formulas, the output series $s^{i}$ coefficient depends only on the $s^{i}$ and lower order coefficients of the input series. Hence, truncation of Taylor series for the purpose of computer implementation does not affect the accuracy of the computed coefficients.

The utility of automatic differentiation is that we can substitute Taylor series such as $W_{<d}(k, s)$ for $(x, y, p_x, p_y)$ in the equations of motion \eqref{H_EOM}, which gives us series in $s$ for $(\dot x, \dot y,\dot p_x, \dot p_y)$. Thus, after overloading the required operators (e.g. arithmetic and power) to accept Taylor series arguments, we can use numerical integration routines with the series as well.\footnote{In fact, any method of overloading the basic operations to accept polynomial arguments could be used in combination with numerical integration here, including methods other than automatic differentiation. As the  implementation of this paper used automatic differentiation, this will be the focus of discussion here.} 
More precisely, consider a Taylor series-valued function of time $V(s,t) = \sum_{j=0}^{\infty}V_{j}(t)s^{j}:\mathbb{R}^{2} \rightarrow \mathbb{R}^{4}$, where $V_{j}(t)$ are its time-varying coefficients. Write $V_{x}(s,t)$, $V_{y}(s,t)$, ${V}_{p_{x}}(s,t)$, and $V_{p_y}(s,t)$ for the $x$, $y$, $p_x$, and $p_y$ components of $V(s,t)$; similarly write $V_{j,x}(t)$, $V_{j,y}(t)$, ${V}_{j,p_{x}}(t)$, and $V_{j,p_y}(t)$ for the components of $V_{j}(t)$. Substituting $V$ in Eq. \eqref{H_EOM} yields a system of differential equations 
\begin{equation} \label{vxdot}  \frac{d}{dt}{V_{x}(s,t)} = \sum_{j=0}^{\infty}\dot V_{j,x}(t)s^{j} = \frac{\partial H}{\partial p_{x}}\Big(V_{x}(s,t),V_{y}(s,t), V_{p_x}(s,t),V_{p_y}(s,t)\Big)  \end{equation}
\begin{equation}  \label{vydot} \frac{d}{dt}{V_{y}(s,t)} = \sum_{j=0}^{\infty}\dot V_{j,y}(t)s^{j} = \frac{\partial H}{\partial p_{y}}\Big(V_{x}(s,t),V_{y}(s,t), V_{p_x}(s,t),V_{p_y}(s,t)\Big)  \end{equation}
\begin{equation}  \label{vpxdot} \frac{d}{dt}{V_{p_x}(s,t)} = \sum_{j=0}^{\infty}\dot V_{j,p_x}(t)s^{j} = -\frac{\partial H}{\partial x}\Big(V_{x}(s,t),V_{y}(s,t), V_{p_x}(s,t),V_{p_y}(s,t)\Big)  \end{equation}
\begin{equation} \label{vpydot} \frac{d}{dt}{V_{p_y}(s,t)} = \sum_{j=0}^{\infty}\dot V_{j,p_y}(t)s^{j} = -\frac{\partial H}{\partial y}\Big(V_{x}(s,t),V_{y}(s,t), V_{p_x}(s,t),V_{p_y}(s,t)\Big)  \end{equation}
Assume that $H$ and its partial derivatives are algebraic functions that are suitable for use with automatic differentiation techniques, e.g. the PCRTBP Hamiltonian Eq. \eqref{pcrtbpH}. Hence, if the $V_{j,x}(t)$, $V_{j,y}(t)$, ${V}_{j,p_{x}}(t)$, and $V_{j,p_y}(t)$ are known for $j \in \mathbb{N}$ and some $t \in \mathbb{R}$, automatic differentiation allows us to simplify the RHS of each of Eq. \eqref{vxdot}-\eqref{vpydot} to a series in $s$. Then, for each of Eq. \eqref{vxdot}-\eqref{vpydot} and $j \in \mathbb{N}$, the $s^{j}$ coefficient $\dot V_{j,x}(t)$, $\dot V_{j,y}(t)$, ${\dot V}_{j,p_{x}}(t)$, or $\dot V_{j,p_y}(t)$ from the LHS must be equal to the $s^{j}$ coefficient of the RHS. In other words, $\dot V_{j,x}(t)$, $\dot V_{j,y}(t)$, ${\dot V}_{j,p_{x}}(t)$, and $\dot V_{j,p_y}(t)$, $j \in \mathbb{N}$, are functions of $ V_{j,x}(t)$, $ V_{j,y}(t)$, ${ V}_{j,p_{x}}(t)$, and $ V_{j,p_y}(t)$, $j \in \mathbb{N}$. This is effectively a system of differential equations for the time-varying Taylor coefficients of $V(s,t)$. Solving Eq. \eqref{vxdot}-\eqref{vpydot}  for the various initial conditions $V(s,0) = W_{<d}(k, s)$, $k=0, 1, \dots, m-1$, we can compute $V(s,\tau(k))=\Phi_{\tau(k)}(W_{<d}(k, s))$ for all desired $k$, which are precisely the Taylor series we needed. 

In summary, we consider the Taylor coefficients of $W_{<d}(k, s)$ as initial state variables to be numerically integrated coefficient by coefficient; propagating by time $\tau(k)$, we get the Taylor coefficients of $\Phi_{\tau(k)}(W_{<d}(k, s))$; the $s^{d}$ coefficient of this gives us $E_{d}(k)$. This approach for numerical integration of Taylor series is called jet transport; see \cite{perezpalau2015} for more details. Truncated Taylor series can be used with jet transport, since the automatic differentiation techniques used to evaluate time derivatives work with truncated series. Note that for degree-$d$ truncated series and our $4$-dimensional phase space, there are $4(d + 1)$ coefficients, which is the required dimension for the numerical integration. 

\subsubsection{Notes about implementation and computation of series $W(k,s)$} 


The parameterization method, automatic differentiation, and jet transport described earlier were implemented for the PCRTBP in a Julia program, leveraging the TaylorSeries.jl \cite{Benet2019}, TaylorIntegration.jl \cite{perezTaylorIntegration}, and OrdinaryDiffEq.jl \cite{RackauckasQing} packages for automatic differentiation and jet transport. The TaylorSeries.jl package already defines a truncated Taylor1 variable type, with built in automatic differentiation routines to operate on them. The OrdinaryDiffEq.jl library, though not originally developed for jet transport, can handle Taylor1 variables as initial conditions when loaded alongside the TaylorIntegration.jl package, propagating them exactly as described in Section \ref{jettransport} on jet transport. Note that OrdinaryDiffEq.jl's built-in DP8 (order 8/5/3 Dormand-Prince Runge-Kutta) adaptive step size integration algorithm was used in this study, not the Taylor integrator of TaylorIntegration.jl. 

Some example manifolds computed using these tools are presented in Section \ref{demoSection}. 
For all example calculations, the final $d$ degree truncated series $W_{\leq d}(k,s)$ were found to satisfy $\Phi_{\tau(k)}(W_{\leq d}(k,s)) - W_{\leq d}(k+1 \mod m, \lambda s) =0$ up to order $s^d$, for each $k = 0, \dots, m-1$; we went up to truncation order $d=20$ or 25 in our series computations. In the $W_{d}(k,s)$ step, we truncate all series at $s^d$ for the automatic differentiation and jet transport steps; this optimizes computational time \& storage requirements.  

As a final remark, note that if $W(k, s)$ is a solution of Equation \eqref{invariancequationfinal}, then so is $W(k, \alpha s)$ for any $\alpha \in \mathbb{R}$. Sometimes, the jet transport integration may struggle to converge due to fast-growing coefficients $W_{j}(k)$ of $W(k, s)$; conversely, fast-shrinking $W_{j}(k)$ can lead to numerical errors in computing $W(k,s)$. In either case, scaling $W(k, s)$ to some $W(k, \alpha s)$ can help. To do this, simply multiply $W_{1}(k)$ by $\alpha$ and then restart the order-by-order algorithm of Section \ref{paramsection}; $\alpha$ should be chosen so that the $W_{j}(k)$ neither grow too rapidly nor shrink to zero. Such an $\alpha$ can be found by running a preliminary calculation of $W(k, s)$, and fitting an exponential growth rate to the resulting coefficients;  alternatively, simple trial and error also often works. 

\subsection{Step 3: fundamental domains of validity for manifold curve parameterizations $W(k,s)$} \label{funDomainSection}

The $d$ degree Taylor parameterizations $W_{\leq d}(k,s)$ of the stable/unstable manifolds of $X(k)$ under the maps $\Phi_{\tau(k)}$ will be much more accurate than their linear approximations by $\bold{\bar v}_{s}(k)$ or $\bold{\bar v}_{u}(k)$. Nevertheless, they will still be inexact due to series truncation error; furthermore, even the exact infinite series $W(k,s)$ satisfying Eq. \eqref{invariancequationfinal} would only be valid for $s$ within some radius of convergence. Hence, one must determine the values of $s \in \mathbb{R}$ for which $W_{\leq d}(k,s)$ accurately represents curves on the stable/unstable manifold. 

Fix an error tolerance, say $E_{tol} = 10^{-5}$ or $10^{-6}$. We now find what \cite{haroetal} calls the fundamental domain of $W_{\leq d}(k, s)$: the largest set $\mathcal D = \{0, 1, \dots, m-1\} \times (-D,D)$ such that for all $(k,s) \in \mathcal{D}$, the error in invariance Eq. \eqref{invariancequationfinal} is less than $E_{tol}$. In other words, we seek the largest $D \in \mathbb{R}^{+}$ such that for all $s$ with $|s| < D$, 
\begin{equation} \max_{k=0, 1, \dots, m-1} \left\|\Phi_{\tau(k)}(W_{\leq d}(k, s)) - W_{\leq d}(k+1 \mod m, \lambda s)\right\| < E_{tol} \end{equation}
The simplest way to find $D$ is to fix $k$ to some value, and then use bisection to find the largest $D_{k}$ such that $\left\|\Phi_{\tau(k)}(W_{\leq d}(k, s)) - W_{\leq d}(k+1 \mod m, \lambda s)\right\| < E_{tol} $ for all $s \in (-D_{k},D_{k})$. After doing this for each value of $k=0,1, \dots, m-1$, $D$ will be the minimum of all the $D_{k}$. 
	


\subsection{Steps 4 and 5: finding the local and then global manifolds $W_p(k,s)$ on the Poincar\'e section  $\Sigma$}  \label{sectionGlobo}

With the series solution $W(k,s)$ to Equation \eqref{invariancequationfinal} computed along with its fundamental domain of validity $|s|<D$, we now seek to use this to find the stable/unstable manifold intersection curves $W^{s}_\Sigma=W^s(\gamma) \cap \Sigma$ or $W^{u}_\Sigma=W^u(\gamma) \cap \Sigma$ with the Poincar\'e section $\Sigma$. Recall from Section \ref{overviewSection} that we seek $W_{p}^{u}, W_{p}^{s}: \{0, 1, \dots, m-1\} \times \mathbb{R} \rightarrow \Sigma$ parameterizing $W^{s}_\Sigma$, $W^{u}_\Sigma$. The equation to be satisfied by $W_{p}^{u}, W_{p}^{s}$ was  
\begin{equation} \label{wpInvariance2} P(W_{p}(k,s)) = W_{p}(k+1 \mod m, \lambda s) \end{equation}
Except at $s=0$, the curves $W(k,s)$ do not satisfy Equation \eqref{wpInvariance2}; they do satisfy the similar Equation \eqref{invariancequationfinal} though. Moreover, as described in Section \ref{paramsection}, the curves $W(k,s)$ lie \emph{near} $\Sigma$ for $s$ near 0, as $W(k,0) = X(k) \in \Sigma$. In practice, this is found to hold for all $s$ within the fundamental domain of $W=W_{\leq d}(k,s)$. 

\subsubsection{Representing the local manifold on $\Sigma$}

Since the curves $W(k,s)$ lie on the 2D stable or unstable manifold of $\gamma$, the desired 1D intersection curves $W_p(k,s)$ of said 2D manifold with Poincar\'e section $\Sigma$ can be found by propagating points from the $W(k,s)$ curves to $\Sigma$. Since each $W(k,s)$ lies near $\Sigma$, only a short forwards or backwards integration will be required. 
More precisely, recall that our Poincar\'e section $\Sigma$ lies in the zero level set of some continuous $\sigma: \mathbb{R}^4 \rightarrow \mathbb{R}$. Also recall $\dot \sigma(\bold{x})= \nabla \sigma(\bold{x}) \cdot [\frac{\partial H}{\partial p_{x}}, \, \frac{\partial H}{\partial p_{y}}, \, -\frac{\partial H}{\partial x}, \, -\frac{\partial H}{\partial y} ] : \mathbb{R}^4 \rightarrow \mathbb{R}$ from Section \ref{sectionSection}, and define maps $P_+, P_- : \mathbb{R}^4 \rightarrow \Sigma$ which propagate points of $\mathbb{R}^4$ forwards or backwards respectively to their first crossing with $\Sigma$. 
All this can be used to define the manifold curves $W_{p}(k,s) \in \Sigma$ in terms of $W(k,s)$, as follows: 
\begin{equation} \label{wpFinal}
W_{p}(k,s)=\begin{cases} 
        W(k,s)  &\text{if $\sigma(W(k,s)) = 0$}, \\
      P_{+}(W(k,s))  &\text{if $\sigma(W(k,s)) \dot \sigma(W(k,s))<0$}, \\
      P_{-}(W(k,s))  &\text{if $\sigma(W(k,s)) \dot \sigma(W(k,s))>0$}
   \end{cases} 
\end{equation} 
Equation \eqref{wpFinal} is valid in the fundamental domain of $W$ defined in Section \ref{funDomainSection}, for $|s|<D$. While $W_p$ thus defined does not give polynomial expressions for the corresponding curves, being defined through the maps $P_{+}, P_{-}$ numerically, Equation \eqref{wpFinal} allows for pointwise evaluation of the $W_p$ curves within $|s|<D$.

To understand Equation \eqref{wpFinal}, note that if at $W(k,s)$ one has $\sigma$ and $\dot \sigma$ with opposite signs and $\sigma$ small, then propagating the point $\mathbf{x}(0)=W(k,s)$ forwards will result in $\sigma(\mathbf{x}(t))\approx \sigma(\mathbf{x}(0)) + \dot \sigma(\mathbf{x}(0)) t$ crossing zero after a short integration. At this time $t$, $\mathbf{x}(t)$ should be in $\Sigma$. Similarly, if $\sigma$ and $\dot \sigma$ have the same signs, then one should propagate backwards in order for $\sigma(\mathbf{x}(t))$ to cross zero (and $\mathbf{x}(t)$ to cross $\Sigma$) very quickly. Since $\Phi_{\tau(k)}(W(k,s)) = W(k+1 \mod m, \lambda s)$ for $|s|<D$, and the $W_p(k,s)$ are defined through brief propagations of the $W(k,s)$ to $\Sigma$, one finds that for $|s|<D$ the resulting $W_p(k,s)$ will satisfy Equation \eqref{wpInvariance2}, as desired. 

\subsubsection{Globalization of the local manifold representation $W_p$} \label{globoSection}

The $W_{p}(k,s)$ of Equation \eqref{wpFinal} locally represents $W^{s}_\Sigma=W^s(\gamma) \cap \Sigma$ or $W^{u}_\Sigma=W^u(\gamma) \cap \Sigma$ near the intersections of $\gamma$ with $\Sigma$ -- or equivalently, the local stable or unstable manifold in $\Sigma$ of the Poincar\'e map periodic orbit $X(k)$. However, so far the function $W_p$ is restricted to the fundamental domain $|s|<D$, yielding only a local manifold, whereas generally the manifold's dynamics away from the base periodic orbit $\gamma$ are of more interest for applications. Thus, to find these stable/unstable manifold curves on $\Sigma$ for $|s| > D$, one seeks to extend the function $W_{p}$ to all $(k,s) \in \{0, 1, \dots, m-1\} \times \mathbb{R}$ in a way that agrees with its known values for $|s|<D$ and satisfies Equation \eqref{wpInvariance2}. This is called globalization.

Repeatedly applying Equation \eqref{wpInvariance2} yields the equation $P^{N}(W_{p}(k,s)) = W_{p}(k+N \mod m, \lambda^N s)$, where the superscript $N \in \mathbb{Z}^{+}$ refers to function composition. This in turn can be rewritten as:
\begin{equation} \label{globos} W_{p}(k,s) = P^{-N}(W_{p}(k+N \mod m, \lambda^N s)) \end{equation}
\begin{equation} \label{globou} W_{p}(k,s) = P^{N}(W_{p}(k-N \mod m, \lambda^{-N} s)) \end{equation}
Eq. \eqref{globos}-\eqref{globou} allow us to extend $W_p(k, s)$ to all $s \in \mathbb{R}$. If $W_p$ is a stable manifold with $|\lambda| <1$, choose $N \geq 0$ such that $| \lambda^{N}s | < D$ and use the known values of $W_{p}$ in its fundamental domain to evaluate Eq. \eqref{globos}. If $W$ is instead an unstable manifold with $|\lambda| > 1$, take $N \geq 0$ so that $| \lambda^{-N}s | < D$ and evaluate Eq. \eqref{globou}. The maps $P^{N}$ and $P^{-N}$ can be computed by numerical integration as described in Section \eqref{sectionSection}. The function $W_p(k,s)$ defined in this manner satisfies Equation \eqref{wpInvariance2} for all $(k,s) \in \{0, 1, \dots, m-1\} \times \mathbb{R}$. Thus, Eq. \eqref{globos}-\eqref{globou} yields the desired global representation of the entire stable or unstable manifold in $\Sigma$. 

\subsubsection{Visualization of $W^{s}_\Sigma$, $W^{u}_\Sigma$ on $\Sigma$} \label{globoViz}


In practice, to use the function $W_{p}(k,s)$ to compute and visualize $W^{s}_\Sigma$ or $W^{u}_\Sigma$ for dynamical analysis, we need enough manifold points on $\Sigma$ to draw the corresponding curves, rather than just a few $W_{p}(k,s)$ values. For this, one should first take an evenly-spaced grid of $M$ $s$-values $\{s_{i}\}_i$ from $-D$ to $D$, and then compute and store $W_{p}(k,s_{i})$ for all $i = 1, \dots, M$ and $k=0, 1, \dots, m-1$ using the known polynomial $W(k,s)$ and Equation \eqref{wpFinal}. Then, repeatedly apply $P$ or $P^{-1}$ to the $W_{p}(k, s_{i})$ to get the points $W_{p}(k, \lambda^{N}s_{i})$ if $|\lambda| > 1$ or $W_{p}(k, \lambda^{-N}s_{i})$ if $|\lambda| < 1$, for all $k=0, 1, \dots, m-1$, $i = 1, \dots, M$ and $N = 0, 1, 2, \dots$ up to some $N_{max} \in \mathbb{Z}^{+}$. Also store the corresponding $(k,s)$ values alongside the $W_p$ points found. One thus will have points $W_p(k,s)$ for all $k=0, 1, \dots, m-1$ on a globalized grid of known $s$ values $s \in \{s_{j,g}\}_j$ stretching (possibly far) outside $(-D, D)$. Hence, for each fixed $k$, one can now plot the points $W_p(k,s_{j,g})$ connected by line segments in order of increasing $s_{j,g}$, which will yield plots of the curves $W_p(k,s)$ on $\Sigma$. Recall that any such curve will in fact lie inside a 2D fixed-energy Poincar\'e section $\Sigma_C \subset \Sigma$, with $C$ being the (constant) value of the energy along the orbit $\gamma$; thus, only two coordinates are required for these plots and subsequent analysis. 

\section{Computation of Heteroclinic Connections } \label{heteroclinicSection}

With the stable/unstable manifolds of periodic orbits parametrized and globalized on the Poincar\'e section $\Sigma$, we now demonstrate how to use these results to find heteroclinic connections between orbits. As in Section \ref{overviewSection}, let $\gamma_{1}$ and $\gamma_{2}$ be two periodic orbits between which we wish to find heteroclinic connections; without loss of generality, we will seek heteroclinics \emph{from} $\gamma_{1}$ \emph{to} $\gamma_{2}$. Let $m_{1}, m_{2} \in \mathbb{Z}^{+}$ denote\footnote{We make a slight abuse of notation here; these newly-introduced $m_1, m_2$ have no relation to the $m_1, m_2$ of Section \ref{modelsection}} the number of intersection points of orbits $\gamma_1$ and $\gamma_2$ respectively with $\Sigma$. Now, suppose we have used the methods of Section \ref{parambigsection} to find $W_{1p}^{u}: \{0, \dots, m_{1}-1\} \times \mathbb{R} \rightarrow \mathbb{R}^{4}$ and $W_{2p}^{s}: \{0, \dots, m_{2}-1\} \times \mathbb{R} \rightarrow \mathbb{R}^{4}$ parameterizing the orbits' unstable and stable manifold $\Sigma$-intersection curves $W^{u}_{1\Sigma}=W^{u}(\gamma_{1}) \cap \Sigma$ and $W^{s}_{2\Sigma}=W^{s}(\gamma_{2}) \cap \Sigma$. Denoting the associated unstable/stable multipliers as $\lambda_{u1}$ and $\lambda_{s2}$ so that $W_{1p}^{u}$ and $W_{2p}^{s}$ satisfy Equation \ref{wpInvariance2}, we thus know that 
\begin{equation} \label{wpInvarianceU1} P(W^{u}_{1p}(k,s)) = W^{u}_{1p}(k+1 \mod m_{1}, \lambda_{u1} s)  \end{equation}
\begin{equation} \label{wpInvarianceS2} P(W^{s}_{2p}(k,s)) = W^{s}_{2p}(k+1 \mod m_{2}, \lambda_{s2} s)  \end{equation}

Now, any heteroclinic connection from $\gamma_{1}$ to $\gamma_{2}$ will correspond to an intersection of the unstable manifold $W^{u}_{1\Sigma}$ with the stable manifold $W^{s}_{2\Sigma}$. Such intersections in turn correspond to solutions of the equation 
\begin{equation} \label{intersectcondition} W^{u}_{1p}(k_{1},s_{1}) = W^{s}_{2p}(k_{2},s_{2}) \end{equation}
where $(k_{1},s_{1}, k_{2},s_{2}) \in \mathbb{Z} \times \mathbb{R} \times \mathbb{Z} \times \mathbb{R}$ are the unknowns to be found. Given such a solution, either side of Equation \eqref{intersectcondition} yields the desired manifold intersection point in $\Sigma \subset \mathbb{R}^{4}$ and resulting heteroclinic trajectory. To solve Equation \eqref{intersectcondition} accurately, we will follow the 3-step approach outlined in Section \ref{overviewSection} for goal 2. We now explain these 3 steps in detail in the following subsections. 

\subsection{Step 1: the method of layers for restricting the connection search} \label{stepLayers}

To solve Equation \eqref{intersectcondition}, it would help to be able to restrict our solution search to only certain regions of $(k_{1},s_{1}, k_{2},s_{2})$ space. To this end, we define the concept of layers. This method was also used in our previous paper \cite{kumar2025siads} on heteroclinics between invariant tori in higher-dimensional systems, and the overall idea is almost exactly the same; thus, the following discussion is largely adapted from that paper. 

Let $|s_{1}| < D_{1}$ and $|s_{2}| < D_{2}$ be fundamental domains (Section \ref{funDomainSection}) for the manifold parameterizations $W_{1p}^{u}$ and $W_{2p}^{s}$, respectively.  Now, define subsets $U_{N}^{+}$, $U_{N}^{-} \subset W_{1\Sigma}^{u}$  and $S_{N}^{+}$, $S_{N}^{-} \subset W_{2\Sigma}^{s}$, $N \in \mathbb{Z}$ as follows:
\begin{gather} \label{layerDef}
U_{N}^{+} = \{ W_{1p}^{u}(k,s) : (k,s) \in  \{0, \dots, m_{1}-1\}  \times [D_{1}\lambda_{u1}^{N-1}, D_{1}\lambda_{u1}^{N}]\} \\
U_{N}^{-} = \{ W_{1p}^{u}(k,s) : (k,s) \in \{0, \dots, m_{1}-1\}  \times [-D_{1}\lambda_{u1}^{N-1}, -D_{1}\lambda_{u1}^{N}]\} \\
S_{N}^{+} = \{ W_{2p}^{s}(k,s) : (k,s) \in  \{0, \dots, m_{2}-1\}  \times [D_{2}/\lambda_{s2}^{N-1}, D_{2}/\lambda_{s2}^{N}]\} \\
\label{layerDef2}
S_{N}^{-} = \{ W_{2p}^{s}(k,s) : (k,s) \in  \{0, \dots, m_{2}-1\}  \times [-D_{2}/\lambda_{s2}^{N-1}, -D_{2}/\lambda_{s2}^{N}]\} 
\end{gather}
Finally, define $U_{N} = U_{N}^{+} \cup U_{N}^{-}$ and $S_{N} = S_{N}^{+} \cup S_{N}^{-}$. We refer to the subsets $U_{N}$ and $S_{N}$ as layers, and to $U_{N}^{+} $, $S_{N}^{+} $ and $U_{N}^{-}$, $S_{N}^{-}$ as positive and negative half-layers, respectively; all these layers (and half-layers) simply consist of $2m_1$ or $2m_2$ (and $m_1$ or $m_2$) finite connected segments of the manifold curves. Based on prior experience, $W_{1p}^{u}(k_{1}, s_{1})$ and $W_{2p}^{s}(k_{2}, s_{2})$  do not intersect for $|s_{1}| < D_{1}$ and $|s_{2}| < D_{2}$; this can also easily be verified for any pair of manifolds by plotting their corresponding curves on $\Sigma$ in 2D coordinates.  Hence, if $W_{1p}^{u}$ and $W_{2p}^{s}$ intersect, it must be that $U_{N_{1}}$ intersects  $S_{N_{2}}$ for some $N_{1}, N_{2} \in \mathbb{Z}^{+}$. 

The most important property of these layers is that due to Equations \eqref{wpInvarianceU1}-\eqref{wpInvarianceS2}, we have that $P(U_{N}) = U_{N+1}$ and $P(S_{N}) = S_{N-1}$; more generally, $P^{K}(U_{N}) = U_{N+K}$ and $P^{K}(S_{N}) = S_{N-K}$ for all $K \in \mathbb{Z}$. This allows us to restrict the heteroclinic connection search to only certain pairs of layers of $W_{1\Sigma}^{u}$ and $W_{2\Sigma}^{s}$. To see this, suppose we are searching for a heteroclinic connection which comes from layer $U_{N_{1}}$ intersecting layer $S_{N_{2}}$ at $\bold{x} \in \Sigma \subset \mathbb{R}^{4}$. Then, since $P(U_{N}) = U_{N+1}$ and $P(S_{N}) = S_{N-1}$, we have that $P(\bold{x})$ must belong to both $U_{N_{1}+1}$ and $S_{N_{2}-1}$. More generally, for all $K \in \mathbb{Z}$, we have that 
\begin{equation} \label{layerintersection} P^{K}(\bold{x}) \in U_{N_{1}+K} \cap S_{N_{2}-K} \end{equation} 
Now, if $N_{1}$ and $N_{2}$ are both odd or both even, using $K = \frac{N_{2} - N_{1}}{2} $ in Equation \eqref{layerintersection} gives us $P^{K}(\bold{x}) \in U_{\tilde N} \cap S_{\tilde N}$, where $ \tilde N \stackrel{\text{def}}{=} \frac{N_{1} + N_{2}}{2}$. On the other hand, if $N_{1}$ and $N_{2}$ are of opposite parity, setting $K = \frac{N_{2} - N_{1}+1}{2}$ in Equation \eqref{layerintersection} gives us $P^{K}(\bold{x}) \in U_{\tilde N} \cap S_{\tilde N -1}$, where $ \tilde N \stackrel{\text{def}}{=} \frac{N_{1} + N_{2} + 1}{2}$. 

When searching for the heteroclinic trajectory which arises due to the manifolds' intersection at $\bold{x}$, it is enough to find any point on the orbit of $\bold{x}$ under the map $P$, including $P^{K}(\bold{x})$ from the preceding analysis. Based on the previous discussion, it is clear that we will find the point $P^{K}(\bold{x})$ if we look for intersections of pairs of layers of form $(U_{N}, S_{N})$ or $(U_{N}, S_{N-1})$ for $N \in \mathbb{Z^{+}}$ (as mentioned earlier, our experience is that the manifolds do not intersect for $|s_{1}| < D_{1}$ and  $|s_{2}| < D_{2}$, so we only consider positive $N$). Since $\bold{x}$ was an arbitrary heteroclinic point, if we search for intersections of pairs of layers of the form just presented above, we will find all possible heteroclinic trajectories. 

As a final note, it is easy to see that if $U_{N_{1}}$ intersects $S_{N_{2}}$, then the resulting heteroclinic connection requires $N_{1}+N_{2}$ Poincar\'e mappings to go from the fundamental domain of $W_{1p}^{u}$ to the domain of $W_{2p}^{s}$. Hence, the layer indices can be thought of as a rough proxy for the connection trajectory time of flight. 

\subsection{Step 2: intersecting line segments approximating manifold curves} \label{stepSegments}

With the connection search now restricted to only layer pairs of form $(U_{N}, S_{N})$ or $(U_{N}, S_{N-1})$, $N \in \mathbb{Z^{+}}$, we next discuss how to find intersections between layers. The first step of this, as mentioned in step 2 of goal 2 from Section \ref{overviewSection}, is to quickly search for and tightly bound  solutions to Equation \eqref{intersectcondition}. This is achieved by representing the manifold curves using line segments and finding their intersections. The final step will be the refinement of any solution bounds found, as will be described in the next Section \ref{stepRefine}. 

The layers defined in Section \ref{stepLayers} correspond to certain ranges of $s_{1}$ and $s_{2}$ values. Recall from Section \ref{globoViz} that for visualization and analysis of a manifold parameterized by $W_{p}(k,s)$ with fundamental domain $|s|<D$, one should compute points $W_{p}(k, \lambda^{N}s_{i})$ if $|\lambda| > 1$ or $W_{p}(k, \lambda^{-N}s_{i})$ if $|\lambda| < 1$ for $s_{i}$ on a fine grid from $-D$ to $D$ inclusive, $k=0, 1, \dots, m-1$, and $N = 0, 1, 2, \dots, N_{max}$. Applying this procedure to our specific $W_{1p}^{u}(k_{1},s_{1})$ and $W_{2p}^{s}(k_{2}, s_{2})$ will generate manifold points corresponding to $k_{1}=0, \dots, m_{1}-1$, $k_{2} = 0, \dots, m_{2}-1$, and $s_{1}$ \& $s_{2}$ belonging to finite sets of values $\{s_{1,j_{1}}\}_{j_{1}} \in [-D_{1}\lambda_{u1}^{N_{max}}, D_{1}\lambda_{u1}^{N_{max}}]$ and $\{s_{2,j_{2}}\}_{j_{2}} \in [-D_{2} \lambda_{s2}^{-N_{max}}, D_{2} \lambda_{s2}^{-N_{max}}]$. Importantly, these sets $\{s_{1,j_{1}}\}_{j_{1}}$ and $\{s_{2,j_{2}}\}_{j_{2}}$ will contain the layer boundary $s$-values of form $s_{1}= \pm D_{1}\lambda_{u1}^{N}$ and $s_{2}=\pm D_{2}\lambda_{s2}^{-N}$ from Equations \eqref{layerDef}-\eqref{layerDef2} for all $N = 0, 1, \dots, N_{max}$. 

For each $M, N < N_{max}$, define finite sets of points $U_{N,f} = \{W_{1p}^{u}(k_{1},s_{1,j_{1}}): s_{1,j_{1}} \in \pm[D_{1}\lambda_{u1}^{N-1}, D_{1}\lambda_{u1}^{N}]\} \subset U_{N}$ and $S_{M,f} = \{W_{2p}^{s}(k_{2},s_{2,j_{2}}): s_{2,j_{2}} \in \pm[D_{2}/\lambda_{s2}^{M-1}, D_{2}/\lambda_{s2}^{M}]\} \subset S_{M}$. Given the previous discussion, both $U_{N,f} $ and $S_{M,f}$ will include all boundary points of the layers $U_{N}$ and $S_{M}$ they are respectively contained in. 
Now, to search for intersections between $U_{N}$ and $S_{M}$, we will approximate the manifold curve segments contained therein by connecting points of $U_{N,f}$ and $S_{M,f}$ with line segments, joining points at each fixed $k_{1}$ or $k_{2}$ in order of consecutive increasing $s_{1,j_{1}}$ or $s_{2,j_{2}}$. This creates linearly-interpolated representations of $U_{N}$ and $S_{M}$ up to their boundaries. Since all manifolds and layers lie within a 2D fixed-energy Poincar\'e section $\Sigma_C \subset \Sigma \subset \mathbb{R}^{4}$, we can thus search for intersections of $U_{N}$ and $S_{M}$ by checking if any segment between points of $U_{N,f}$ intersects any segment connecting points of $S_{M,f}$ when \emph{projected into 2D section coordinates}.


It is very simple to check whether two line segments intersect in 2D. Given endpoints $\bold{x}_{1}, \bold{x}_{2} \in \mathbb{R}^{2}$ for segment 1 and $\bold{y}_{1}, \bold{y}_{2} \in \mathbb{R}^{2}$ for segment 2, one just needs to solve the $2 \times 2$ linear system of equations
\begin{equation} \label{precise} \bold{x_{1}} + (\bold{x_{2}}-\bold{x_{1}})a = \bold{y_{1}} + (\bold{y_{2}}-\bold{y_{2}})b \end{equation}
for $a,b \in \mathbb{R}$ and check whether $a,b \geq 0$, $a \leq 1$, and $b \leq 1$. If so, it is easy to see that the LHS and RHS of Equation \eqref{precise} are themselves equal to the intersection point of the line segments. Since this is a $2\times2$ system, the solution of Equation \eqref{precise} can easily be written explicitly in terms of the coordinates of $\bold{x}_{1}, \bold{x}_{2}, \bold{y}_{1},$ and $\bold{y}_{2}$, making it conducive to implementation even on a GPU. Indeed, the search for intersections between $U_{N,f}$ segments and $S_{M,f}$ segments is embarrassingly parallel and benefits greatly from parallelization. 


If an intersection is found between a segment connecting $U_{N,f}$ points and a segment between $S_{M,f}$ points, we record 8 corresponding parameter values: the $(k_1, s_1)$ values for the (two) endpoints of the $U_{N,f}$ segment and the $(k_2, s_2)$ values for the endpoints of the $S_{M,f}$ segment. The two $U_{N,f}$ endpoints will share the same $k_1$ value but have consecutive $s_1$ values ($s_{1,j_1}$ and $s_{1,j_1+1}$). Similarly, the $S_{M,f}$ endpoints share the same $k_2$ value with consecutive $s_2$ values ($s_{2,j_2}$ and $s_{2,j_2+1}$). These parameter values provide the initial bounds for a more precise search. While the solution to Equation \ref{precise} can be used to interpolate these $s_{1}$ and $s_{2}$ values and estimate a solution to Equation \eqref{intersectcondition}, we will instead use the saved $k_{1}, s_{1}, k_{2}$, and $s_{2}$ values to initiate a bisection process, as described in Section \ref{stepRefine}. This will yield a much more accurate heteroclinic solution. 

This search procedure of checking intersections of $U_{N,f}$ and $S_{M,f}$ segments is done for $N=0, 1, \dots, N_{max}$, at each of whose values we only consider $M=N$ and $N-1$ given the discussion of Section \ref{stepLayers}.


\begin{remark} Since one can only compute finitely many manifold points, one practical challenge that can arise in this step is that parts of the line segment-based manifold curves can become visually highly ``discontinuous''. Segment endpoints can grow very far apart despite being propagated from nearby points or may be scattered in directions very different from others at nearby $s$-values. For example, this occurs in the {PCRTBP} when manifold trajectories pass very close to the smaller massive body. In such cases, ``discontinuous segments'' should be excluded from the intersection search. A simple heuristic for detecting these cases is to check if the distance between consecutive points $\| W_{1p}^{u}(k_{1}, s_{1,j_{1}+1}) - W_{1p}^{u}(k_{1}, s_{1,j_{1}}) \|$ is unusually large or significantly greater than the distance between the previous points $\| W_{1p}^{u}(k_{1}, s_{u,j_{1}}) - W_{1p}^{u}(s_{u,j_{1}-1}) \|$  (similarly for $W_{2p}^{s}$). 
\end{remark}

\subsection{Step 3: refinement of approximate solutions} \label{stepRefine}

With the algorithm of Section \ref{stepSegments} serving to find intersecting segments of points from $W_{1p}^{u}$ and $W_{2p}^{s}$, as well as the $k_{1}$, $s_{1}$, $k_{2}$, and $s_{2}$ values corresponding to their endpoints, the final step of the heteroclinic search uses this information to find a highly accurate solution $(k_{1},s_{1}, k_{2},s_{2}) \in \mathbb{Z} \times \mathbb{R} \times \mathbb{Z} \times \mathbb{R}$ to the heteroclinic Equation \eqref{intersectcondition} using bisection. This step no longer relies on the line segments representation of the manifolds from the previous step; instead, we henceforth use Eqs. \eqref{globos}-\eqref{globou} to evaluate the manifold points. 

Suppose that in the previous step, an intersection was found between the segments with endpoints $\{W_{1p}^{u}(k_{1}, s_{1,j_{1}}), W_{1p}^{u}(k_{1}, s_{1,j_{1}+1})\}$ and $\{W_{2p}^{s}(k_{2}, s_{2,j_{2}}), W_{2p}^{s}(k_{2}, s_{2,j_{2}+1})\}$. Then, the corresponding exact solution to Equation \eqref{intersectcondition} should have $s_{1} \in [s_{1,j_{1}}, s_{1,j_{1}+1}]$ and $s_{2} \in [s_{2,j_{2}}, s_{2,j_{2}+1}]$. 
Let $a_{1,0}= s_{1,j_{1}}$, $b_{1,0}= s_{1,j_{1}+1}$, $a_{2,0}= s_{2,j_{2}}$, and $b_{2,0}= s_{2,j_{2}+1}$. We now use these definitions to start a recursive iteration generating sequences $a_{1,i}, b_{1,i}, a_{2,i}$, and $b_{2,i}$, $i = 0, 1, 2, \dots$ which will converge to the desired $s_{1}$ and $s_{2}$ as follows:
		
	\begin{enumerate}
	\item Let $c_{1,i}=\frac{a_{1,i}+ b_{1,i}}{2}$, and compute $W_{1p}^{u}(k_{1}, c_{1,i})$ using Equation \ref{globou} as described in Section \ref{globoSection}. 
	\item Let $c_{2,i}=\frac{a_{2,i}+ b_{2,i}}{2}$, and compute $W_{2p}^{s}(k_{2}, c_{2,i})$ using Equation \ref{globos} as described in Section \ref{globoSection}. 
	\item Denote the line segments with endpoints $\{W_{1p}^{u}(k_{1}, a_{1,i}), W_{1p}^{u}(k_{1}, c_{1,i})\}$, $\{ W_{1p}^{u}(k_{1}, c_{1,i}), W_{1p}^{u}(k_{1}, b_{1,i})\}$, $\{W_{2p}^{s}(k_{2}, a_{2,i}), W_{2p}^{s}(k_{2}, c_{2,i})\}$, and $\{ W_{2p}^{s}(k_{2}, c_{2,i}), W_{2p}^{s}(k_{2}, b_{2,i})\}$ as $L_{11}, L_{12}, L_{21}$, and $L_{22}$ respectively. 
	Check if either of $L_{11}$ or $L_{12}$ intersects either of $L_{21}$ or $L_{22}$ when projected onto 2D coordinates for $\Sigma$. 
	\item If no intersection was found in the previous step, exit iteration and indicate that no solution exists. If an intersection was found, set $a_{1,i+1}$, $b_{1,i+1}$, $a_{2,i+1}$, and $b_{2,i+1}$ to the $s_{1}$ and $s_{2}$-values of the intersecting segments' endpoints, and return to step 1 for next recursive step. 
   \item End recursive bisection when $|a_{1,i}-b_{1,i}|$ and $|a_{2,i}-b_{2,i}|$ are small enough. 
   	\end{enumerate}
Note that $k_{1}$ and $k_{2}$ do not change during the bisection. If the above process converges in the sense of $|a_{1,i}-b_{1,i}|$ and $|a_{2,i}-b_{2,i}|$ becoming less than some small tolerance, then the desired $s_{1}$ and $s_{2}$ will be very well-approximated by the final values of $\frac{a_{1,i}+ b_{1,i}}{2}$ and $\frac{a_{2,i}+ b_{2,i}}{2}$, respectively. The resulting $(k_{1},s_{1}, k_{2},s_{2}) $ should satisfy the equation $W^{u}_{1p}(k_{1},s_{1}) = W^{s}_{2p}(k_{2},s_{2})$ up to a very small error, whose RHS or LHS thus yields the desired intersection between $W^{u}_{1\Sigma}$ and $W^{s}_{2\Sigma}$ and the corresponding heteroclinic connection as well. 

\begin{remark} 
For the refinement process described here, a Newton method could also have been implemented with the aid of the truncated Taylor series used to define the manifold parameterizations $W_{1p}^{u}$ and $W_{2p}^{s}$. However, note that differentiating any such parameterization $W_{p}(k,s)$ with respect to $s$  requires differentiating Equation \eqref{globos} or \eqref{globou}. This in turn requires differentiating the Poincar\'e map $P^{N}$ or $P^{-N}$, as well as differentiating $W_{p}$ at a smaller $s$ value through Equation  \eqref{wpFinal}, which then requires differentiation of $P^{+}$ or $P^{-}$ as well as the polynomial $W(k,s)$. For the sake of simplicity in implementation, we choose to use bisection instead. 
\end{remark}

\section{ Example Application: Studying Mean Motion Resonances in the Planar CRTBP} \label{demoSection}

The methodology developed in Sections \ref{parambigsection}-\ref{heteroclinicSection} is general, and can be applied to any 2 DOF Hamiltonian system satisfying the assumptions of Section \ref{settingSection}. It was successfully used in a number of previous PCRTBP studies  \cite{rawat2025preprint, kumar2024aug, kumar2024iac} to compute Poincar\'e map stable/unstable manifolds of various resonant periodic orbits in Earth-Moon and Uranus-Oberon systems. Heteroclinics between Earth-Moon resonant orbits were also successfully computed. After a brief discussion of PCRTBP-specific aspects of the method implementation, in this section we include a few examples of the manifolds and heteroclinics computed in those studies, illustrating the effectiveness and utility of the developed parameterization method for orbital mechanics investigations. Full discussion of the dynamical implications of these manifolds' behaviors can be found in the papers \cite{kumar2024aug} (on the Uranus-Oberon system) and \cite{rawat2025preprint, kumar2024iac} (on the Earth-Moon PCRTBP). 

As a note on computational performance, the studies discussed in Sections \ref{UranusSection}-\ref{MoonSection} were carried out on a 2019-era Mac laptop with an Intel i9 8-core CPU and an AMD Radeon Pro 5500M GPU. All algorithms were implemented in the Julia programming language. The calculation of adapted frames, degree-20 manifold parameterizations, and fundamental domains (Sections \ref{step1}-\ref{funDomainSection}) took less than 5 seconds total per manifold in all cases. The globalization/visualization step of Section \ref{globoViz} takes longer due to the many numerical integrations required to compute enough manifold points for plotting; however, this step still generally takes well under 30 seconds, and often less than 15. The computation of heteroclinic connections depends on the number of manifold intersections found; the line segment intersection step 2 (Section \ref{stepSegments}) took less than 0.1 seconds when parallelized on the GPU using OpenCL.jl, while the refinement of each approximate heteroclinic intersection found (step 3, Section \ref{stepRefine}) took between 0.1 and 2 seconds per intersection. 

\subsection{Method implementation in the planar CRTBP}

Recall from Sections \ref{delaunaySection} and \ref{mmrsection} that $\ell=0$ (periapse) is a good Poincar\'e section $\Sigma$ for analyzing PCRTBP interior mean motion resonances (MMRs), and  $\ell= \pi$ (apoapse) is a good section for exterior MMR study. Poincar\'e maps on either $\Sigma$ are computed as described in Section \ref{sectionSection}. An unstable $m$:$n$ MMR periodic orbit will cross either section $m$ times in one period, which will be the number of distinct points $X(k)$ and curves $W(k,s)$, $W_p(k,s)$ to be calculated on that orbit's stable or unstable manifold; we will compute and plot both manifolds for each orbit considered. 

Since each fixed-energy section $\Sigma_C \subset \Sigma$ containing an orbit's Poincar\'e map stable/unstable manifolds is 2D, only two coordinates are needed for visualization and analysis of manifold (and all other) dynamics inside $\Sigma_C$ at any given $C$. All manifold computations are carried out in Cartesian coordinates $(x,y,p_x,p_y) \in \mathbb{R}^4$, after which heteroclinic computations are done by projecting the manifolds into $(x,y)$ coordinates for $\Sigma_C$. While these calculations use Cartesian coordinates,  for visualization we found that the synodic Delaunay coordinates of Section \ref{delaunaySection} yield more illuminating plots. Thus, the calculated $W_p(k,s)$ Cartesian points are converted and plotted in synodic Delaunay coordinates $(L, g)$ in the figures to follow. 

\subsection{Study of MMR overlap in the Uranus-Oberon PCRTBP} \label{UranusSection}

In the paper \cite{kumar2024aug}, the methods of Section \ref{parambigsection} were used to compute stable/unstable manifolds of 3:4, 4:5, and 5:6 exterior and 4:3, 5:4, and 6:5 interior unstable resonant periodic orbits for Poincar\'e maps of the Uranus-Oberon PCRTBP. A few of these stable/unstable manifolds are shown in Figure \ref{fig::ext1} for the exterior MMR manifolds, and Figure \ref{fig::int1} for the interior MMR manifolds. 
\begin{figure}
\includegraphics[width=0.49\columnwidth]{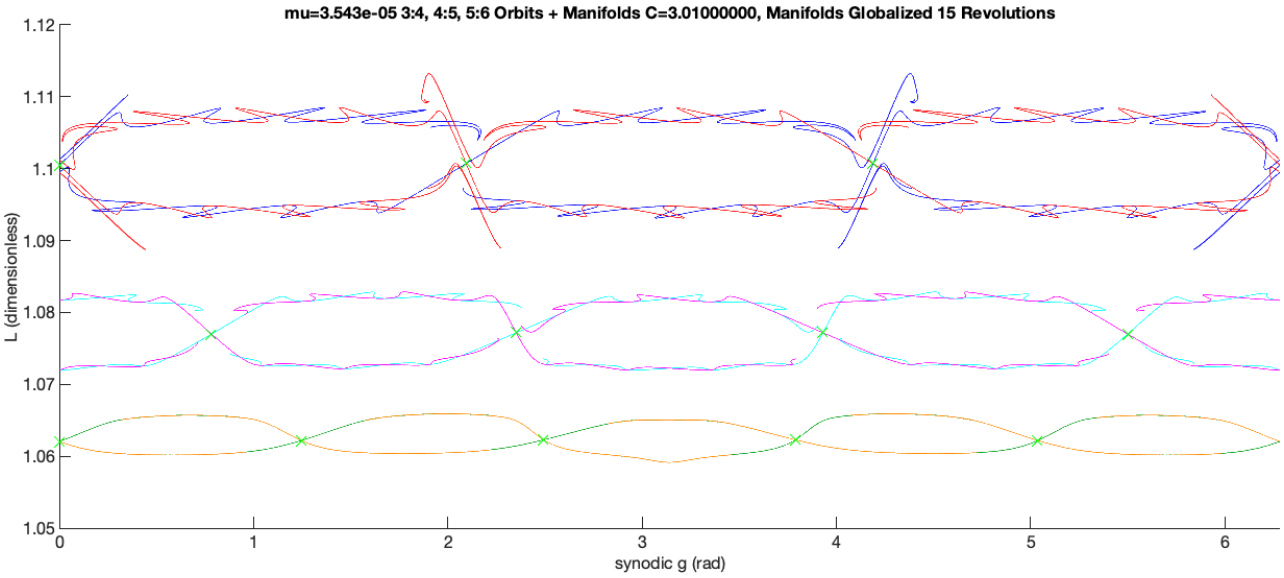}
\includegraphics[width=0.49\columnwidth]{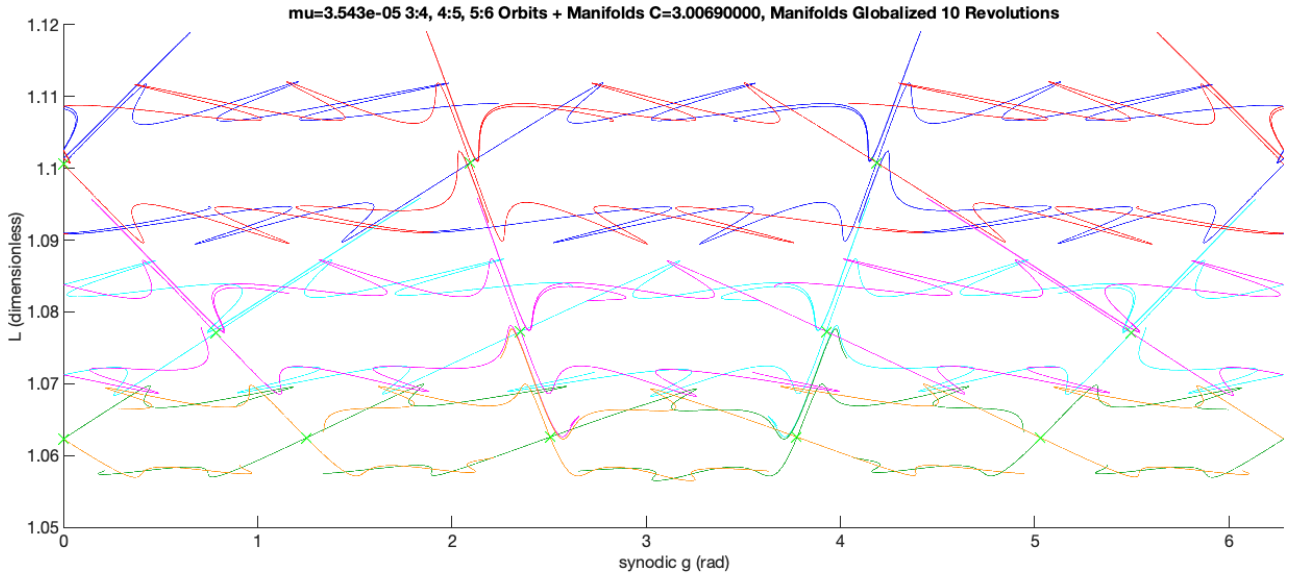}
\caption{ \label{fig::ext1} 3:4 (red/blue), 4:5 (magenta/cyan), and 5:6 (orange/green) unstable and stable manifolds respectively on apoapse section $\ell = \pi$, Jacobi $C=3.0100$ (left) and $C=3.0069$ (right), plotted in synodic Delaunay coordinates $(L,g)$ (from \cite{kumar2024aug})}
\end{figure}
\begin{figure}
\includegraphics[width=0.49\columnwidth]{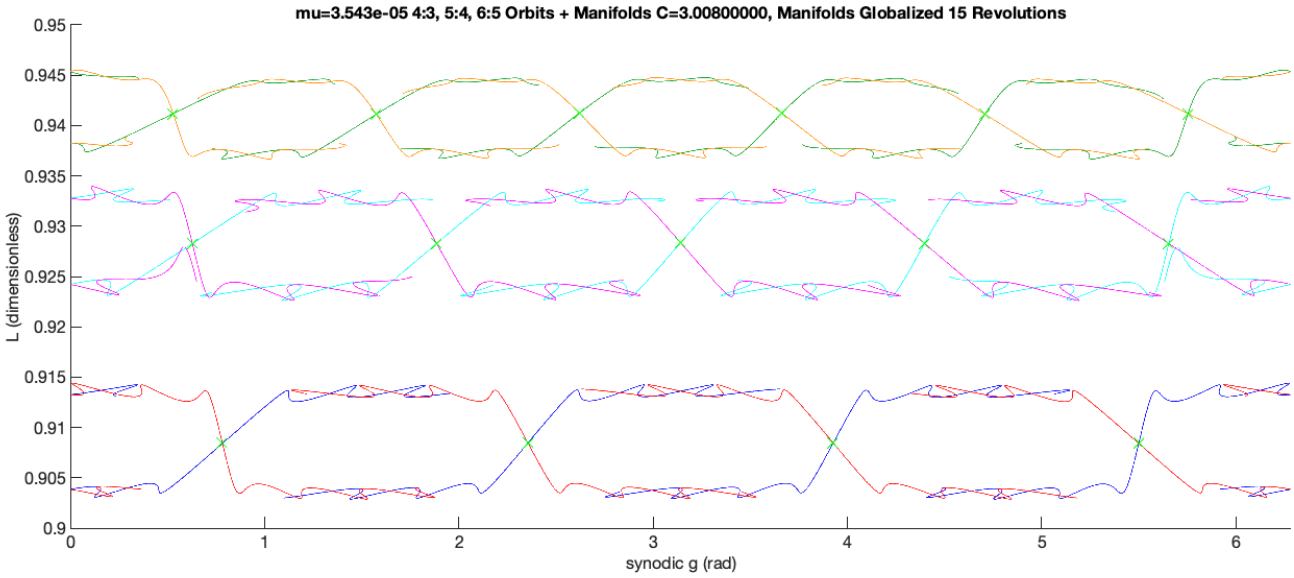}
\includegraphics[width=0.49\columnwidth]{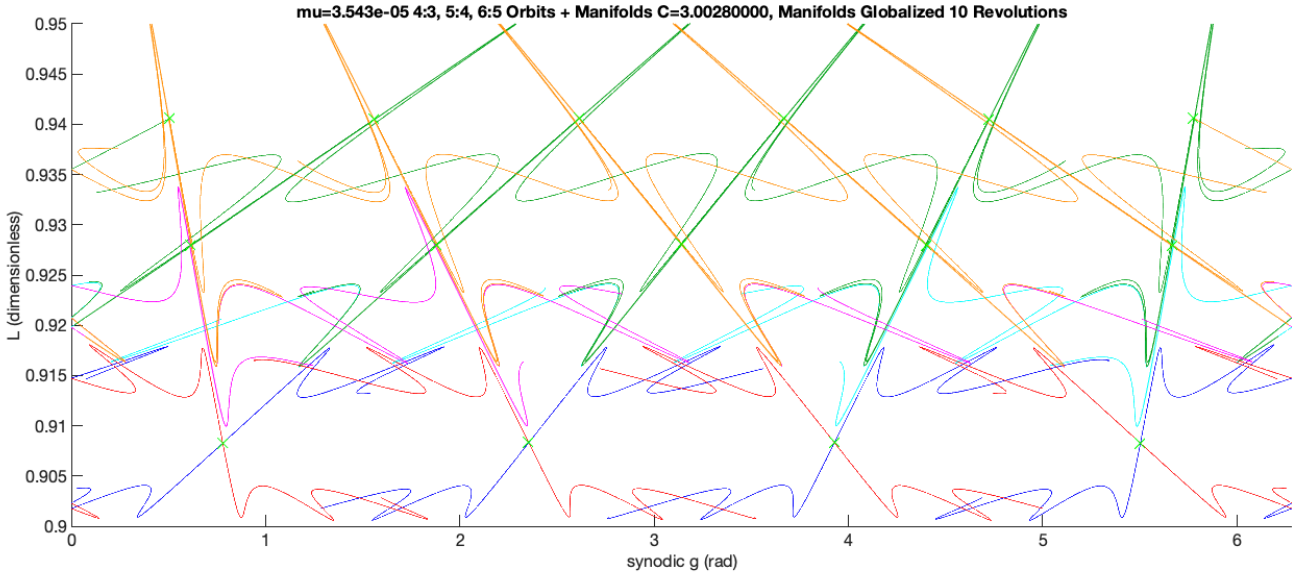}
\caption{ \label{fig::int1} 4:3 (red/blue), 5:4 (magenta/cyan), and 6:5 (orange/green) unstable and stable manifolds respectively on periapse section $\ell = 0$,  Jacobi C=3.0080 (left) and C=3.0028 (right), plotted in synodic Delaunay coordinates $(L,g)$ (from \cite{kumar2024aug})}
\end{figure}
The use of an $\ell=\pi$ section in Figure \ref{fig::ext1} and $\ell=0$ section in Figure \ref{fig::int1}, along with the choice of synodic Delaunay coordinates $(L,g)$ for visualization, has a number of benefits. First of all, the section's  better transversality to the PCRTBP flow ensures that the manifold curves seen in both figures are continuous; compare this with, for example, the jumps seen in the manifold curves of Figure 5 of the paper \cite{kumar2021journal} which used a $y=0$ section. In combination with the use of $(L,g)$ action-angle coordinates, the manifold curves' continuity allows for clear identification of lobes \cite{KoLoMaRo} as well as of intersection points between stable and unstable manifolds. The use of the periapse and apoapse Poincar\'e sections here was thus highly useful; the ability of the method of Section \ref{parambigsection} to accurately compute these long manifold curves in such Poincar\'e sections -- even though the studied unstable resonant periodic orbits each had multiple intersections with it -- aids in enabling this analysis. 

Through computing these manifolds and visually detecting their intersections (or lack thereof), one can determine the energy values at which heteroclinic connections occur between various unstable resonant orbits - a phenomenon is known as \emph{MMR overlap}. For instance, notice that the left plots of Figures \ref{fig::ext1} and \ref{fig::int1}, at Jacobi constants $C=3.0100$ and $3.0080$ respectively, display no overlap between the studied MMRs. The right plots at lower Jacobi constants do. MMR overlap enables global transport across the system phase space, by creating heteroclinic trajectories with large natural changes in $L$. Though not computed in this study, such trajectories are useful for for zero-fuel spacecraft transfers between orbits at different semi-major axis values $a = L^2 / \mu$.  Refer to the paper \cite{kumar2024aug} for more analysis of these behaviors and applications. 

\subsubsection{Accuracy comparison with linear eigenvector manifold approximations}

The variety of periodic orbit manifolds computed in \cite{kumar2024aug} provided a useful set on which to compare accuracy of this paper's nonlinear manifold parameterizations with the commonly-used eigenvector (linear) manifold approximations. This comparison can be carried out by comparing the fundamental domain of $W(k,s)$ (defined in Section \ref{funDomainSection}) with that of the Taylor series $W(k,s)$ truncated to order $s^1$ (which just yields the eigenvector approximation of the manifolds). The study \cite{kumar2024aug} computed stable/unstable manifolds for 3:4, 4:5, and 5:6 exterior and 4:3, 5:4, and 6:5 interior unstable resonant periodic orbits at Jacobi constant values of 3.00 to 3.01 in increments of 0.0001, with Taylor series computed to order $s^{20}$. Thus, fundamental domains of both the degree-20 series as well as of their linear truncations were computed for all these MMRs and Jacobi constants. The ratios of full degree-20 series domains to those of the linear truncations were then calculated, with the resulting data summarized (over all Jacobi constants) in Table \ref{comparisonTable}. 

\begin{table}[]
\begin{centering}
\begin{tabular}{| c | c | c | c | c |}
\hline
MMR     & Minimum     & Maximum  & Mean & Median  \\ \hline
3:4 (exterior) & 179.03 & 2562.78 & 808.37 & 525.38 \\ \hline
4:5   (exterior)  &  158.49  & 3841.14 & 1320.07 & 784.00 \\ \hline
5:6   (exterior)  &   160.97  & 4550.00  & 1533.86 & 496.83  \\ \hline
4:3   (interior)  &  295.18  & 1283.00  & 566.65 &  492.42 \\ \hline
5:4  (interior) &  237.96  & 4268.00  &  1465.14 &  814.88 \\ \hline
6:5  (interior) &  215.96 &   4248.82 &  1448.11 &  908.30 \\
\hline
\end{tabular}
\caption{ \label{comparisonTable} Summary of ratios of degree-20 series fundamental domains with linear truncation fundamental domains for stable/unstable manifolds across all Jacobi constants 3.00 to 3.01, manifolds computed in the paper \cite{kumar2024aug} }
\end{centering}
\end{table}

As can be seen in the table, the method developed in this paper improves accuracy of the computed manifolds very significantly versus the usual linear approximations by eigenvectors. The fundamental domains of the degree-20 series manifold parameterizations were mostly on the order of 0.1, with a few on the order of 0.01; in contrast, the domains of the linear approximations were generally on the order of $10^{-4}$ to $10^{-5}$. The lowest ratio of fundamental ratio improvement in Table \ref{comparisonTable} is 158.49, with the mean improvement ratio across all computed manifolds being 1190.37 -- a major difference, clearly demonstrating the increase in accuracy enabled by the methods of this paper. This accuracy improvement means that less numerical integration is required when globalizing the manifolds, thus also saving computation time.

\subsection{Study of MMRs and resonant transfers in the Earth-Moon PCRTBP} \label{MoonSection}

In \cite{rawat2025preprint, kumar2024iac}, the tools of Section \ref{parambigsection} were used to compute stable/unstable manifolds of 4:1, 3:1, and 2:1 interior unstable resonant periodic orbits for periapse ($\ell = 0$) Poincar\'e maps of the Earth-Moon PCRTBP. Two 4:1 orbits' stable and unstable manifolds are plotted in $(L,g)$ synodic Delaunay coordinates in Figure \ref{fig:41manifolds}; they strongly resemble the shape of separatrices seen in phase portraits of the mathematical pendulum, as expected from Hamiltonian perturbation theory \cite{morbyBook} given that synodic Delaunay coordinates are action-angle variables for the $\mu=0$ PCRTBP. The use of the $\ell=0$ section with accurate manifold parameterizations, enabled by the methods of Section \ref{parambigsection}, allows for computation of very long, continuous manifold curves  (e.g. the right plot of Figure \ref{fig:41manifolds}). The lack of ``excursions'' of these curves to $L$ values far from 0.62, despite a high level of globalization, indicates a lack of heteroclinics between 4:1 and other major MMRs. 

\begin{figure}
\begin{centering}
\includegraphics[width=0.49\columnwidth]{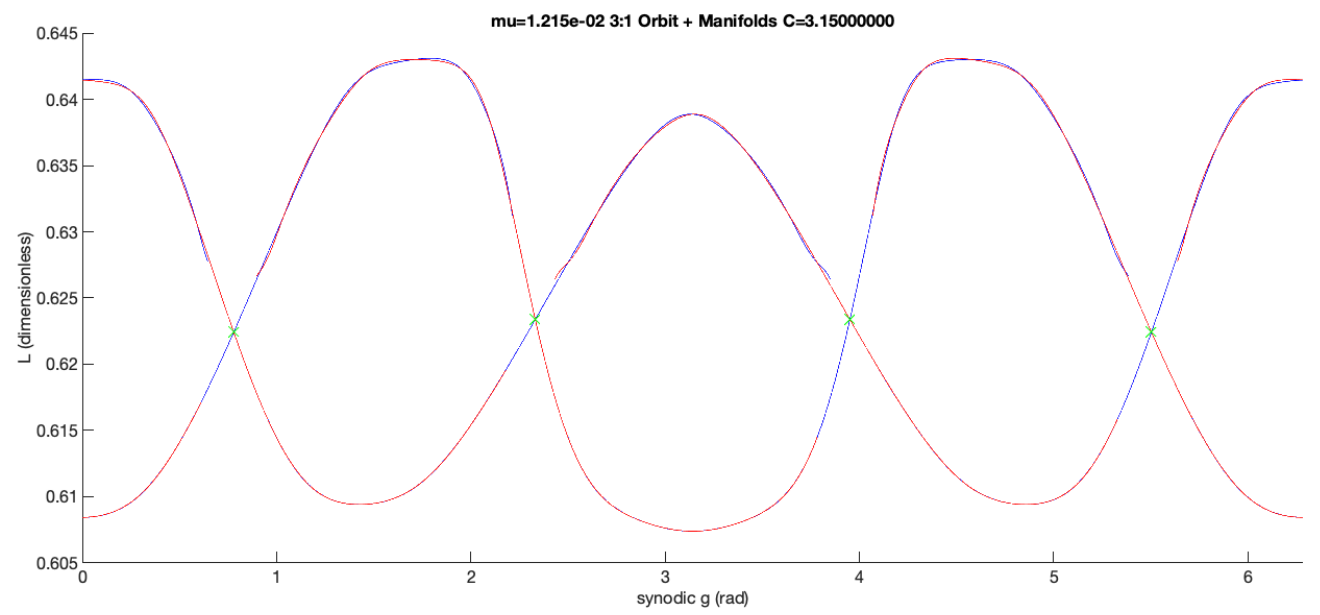}
\includegraphics[width=0.49\columnwidth]{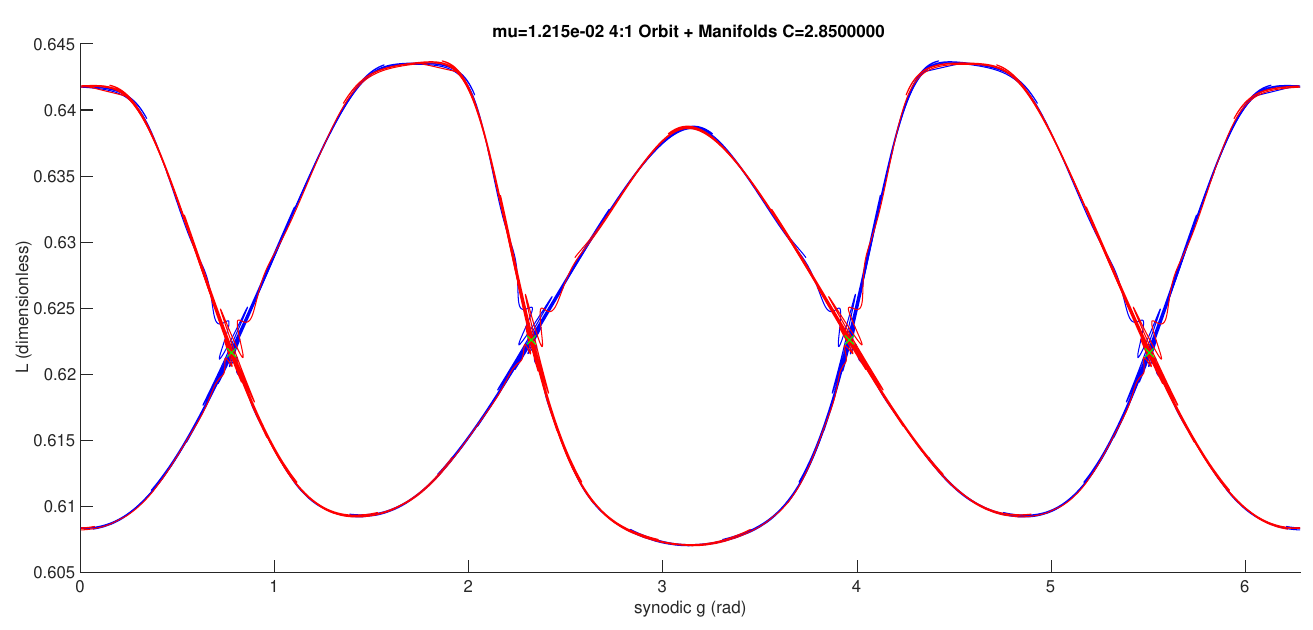}
\caption{ \label{fig:41manifolds} 4:1 stable/unstable manifolds (blue/red), perigee section, for $C = 3.15$ (left), $2.85$ (right) (from \cite{kumar2024iac})}
\end{centering}
\end{figure}

\begin{figure}
\begin{centering}
\includegraphics[width=0.495\columnwidth]{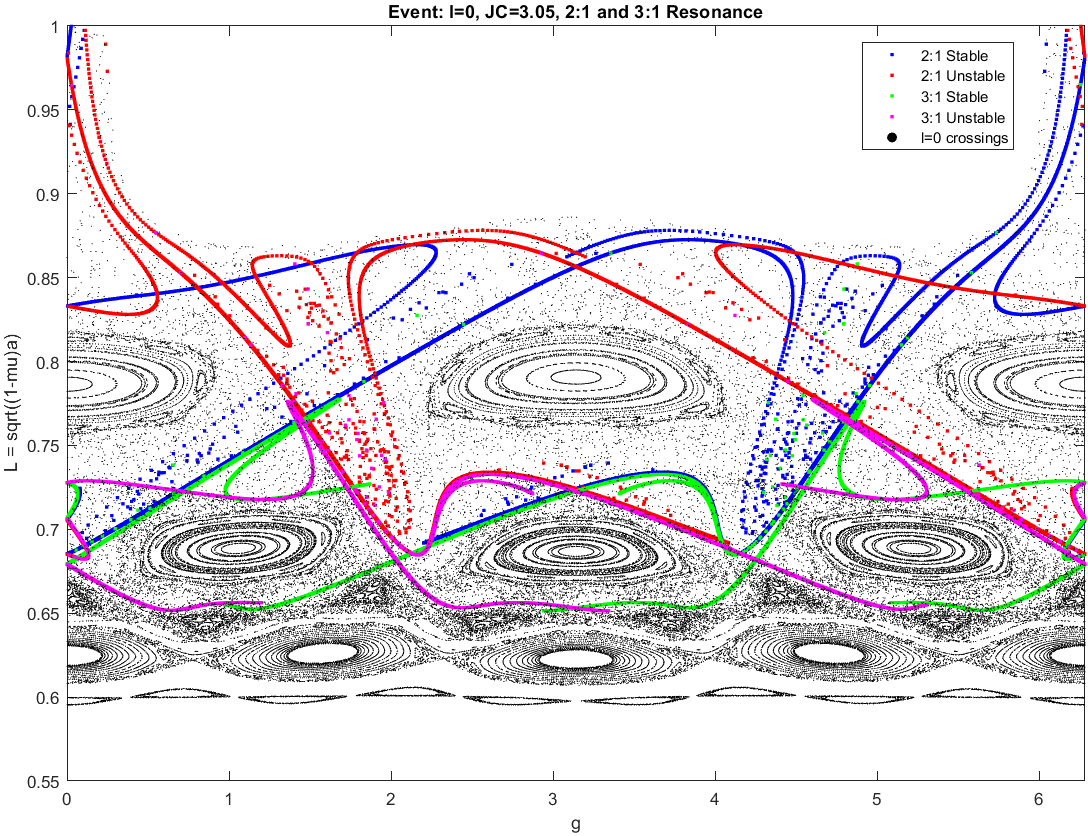}
\includegraphics[width=0.495\columnwidth]{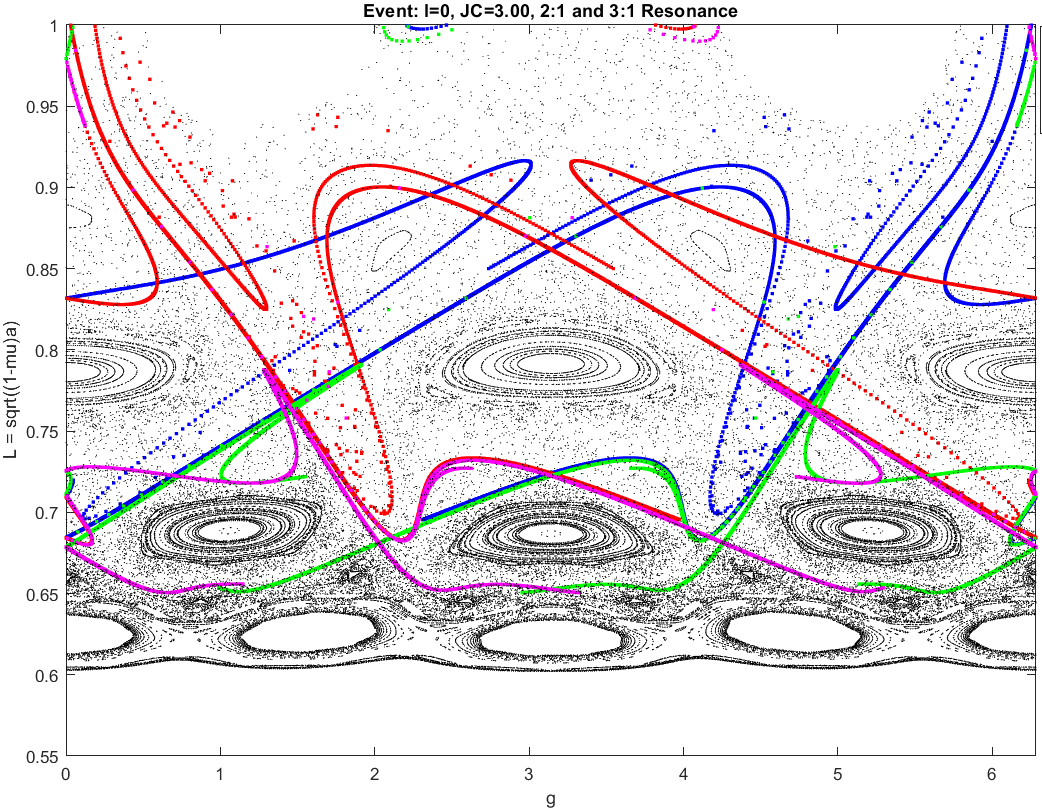}
\caption{ \label{fig:3121manifolds} 3:1 (green/magenta) and 2:1 (blue/red) orbit stable/unstable manifolds on perigee Poincar\'e section along with other propagated trajectories in gray, for Jacobi constants $C = 3.05$ (left), $3.00$ (right) (from \cite{rawat2025preprint})}
\end{centering}
\end{figure}

Figure \ref{fig:3121manifolds} displays stable/unstable manifolds similarly computed for the 3:1 (green/magenta respectively) and 2:1 (blue/red) unstable MMR orbits at Jacobi constants $C= 3.05$ and 3.00. Points from a grid of trajectories not belonging to the manifolds, propagated to the same $\ell=0$ section at the same $C$ values, are shown in gray. The stable/unstable manifolds surround zones of stable MMR librational tori as expected, further confirming the accuracy of the computations. Heteroclinic intersections between 3:1 unstable and 2:1 stable manifolds (and vice versa) are clearly visible, corresponding to initial conditions for fuel-free trajectories between 3:1 and 2:1 unstable resonant orbits. These heteroclinic trajectories were successfully and accurately computed using the tools of Section \ref{heteroclinicSection}; one such trajectory is shown in both rotating (left) and inertial (right) reference frames in Figure \ref{fig:heteroTraj}.
\begin{figure}
\centering
\includegraphics[width=0.495\columnwidth]{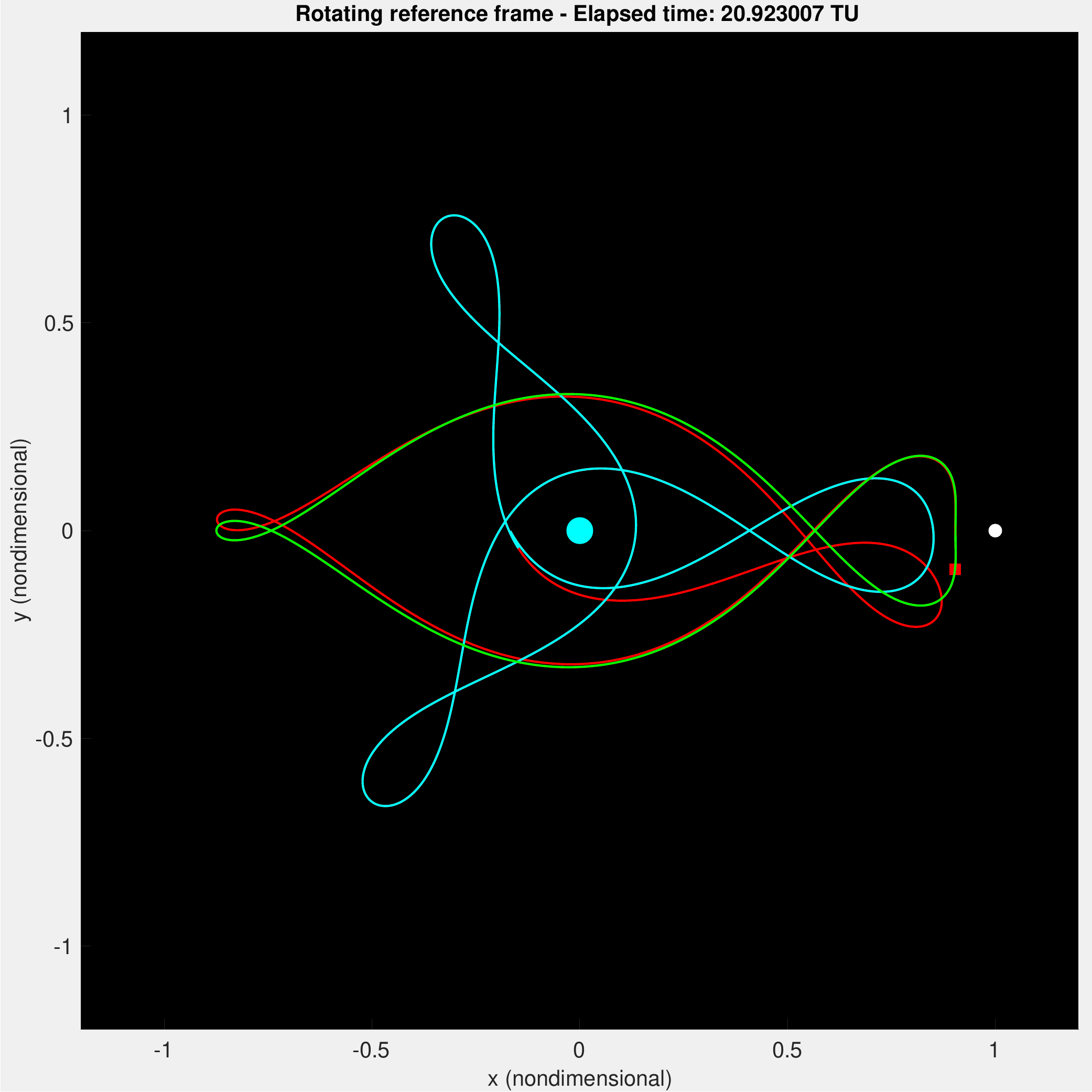}
\includegraphics[width=0.495\columnwidth]{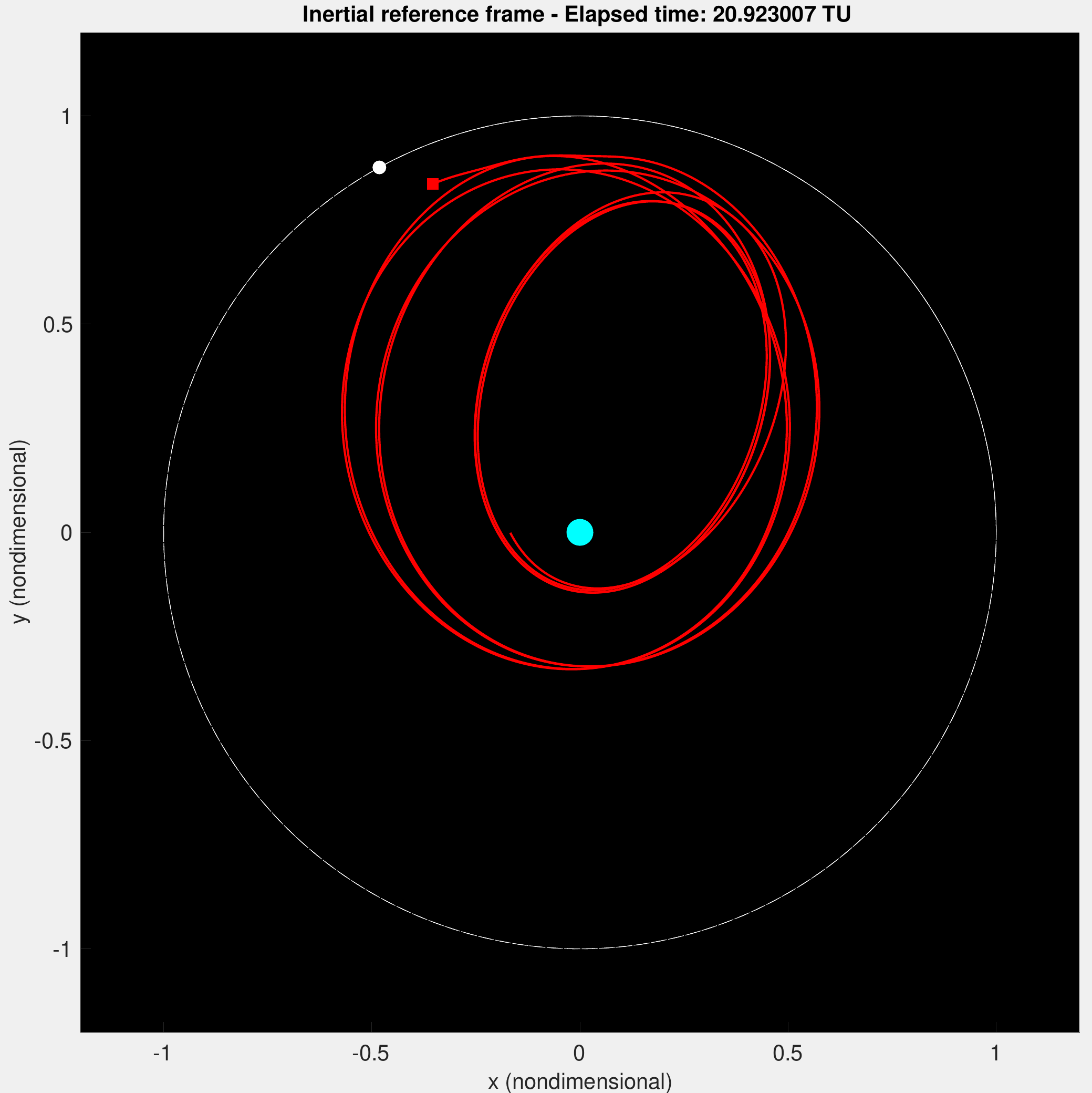}
\caption{ \label{fig:heteroTraj} 3:1 to 2:1 Lunar MMR Heteroclinic Connection Trajectory at Jacobi constant $C=3.05$, with near-3:1 and near-2:1 portions highlighted in cyan and green respectively in rotating frame plot (left). Same trajectory in inertial frame on right.}
\end{figure}
The increase of orbit ellipse semi-major axis along the trajectory is clearly seen in the inertial frame plot, indicating the utility of such trajectories for spacecraft orbit transfers. The calculation of this trajectory was enabled by  the accurate parameterizations of the 3:1 and 2:1 orbits' manifolds due to Section \ref{parambigsection}, combined with the fast heteroclinic solving algorithm of Section \ref{heteroclinicSection}. For further details on Earth-Moon MMR dynamics, the reader is referred to our papers \cite{rawat2025preprint, kumar2024iac}.

\section{Conclusions}

In this work, a parameterization method was developed for the computation of accurate Poincar\'e map stable \& unstable manifolds of periodic orbits in 2 DOF Hamiltonian systems, thus also enabling accurate computation of heteroclinic trajectories between them. The method first finds high-degree polynomial approximations of curves lying on periodic orbit stable/unstable manifolds; then, propagating points from these curves to the section yields the desired Poincar\'e map-invariant manifolds. The method avoids the need to compose Poincar\'e maps with polynomials (as the method of \cite{gonzalezJames} would require), requiring fixed-time propagations instead. The method also works in the case that the flow's periodic orbit intersects the Poincar\'e  section in multiple distinct points, unlike the previous work \cite{kumar2021journal}. It hence allows for accurate computation of stable/unstable manifolds and heteroclinic connections for a much wider variety of Poincar\'e sections, some of which may be better suited to the Hamiltonian dynamical system of interest. 

The methodology developed was successfully implemented and used for a number of studies \cite{rawat2025preprint, kumar2024aug, kumar2024iac} in the planar circular restricted 3-body problem, demonstrating its utility for practical analysis of real-world dynamical systems. The algorithms enable computation of accurate (nonlinear) stable/unstable manifold parameterizations on periapse and apoapse Poincar\'e sections that have good transversality properties to the PCRTBP flow, improving upon the linear methods and fixed Cartesian coordinate sections used in most preceding studies. Major improvements in accuracy were demonstrated as compared to linear eigenvector approximations. The calculations of stable/unstable manifolds take only a few seconds on a laptop for each resonant periodic orbit, whose manifold curves can then be plotted using just two coordinates in a 2D fixed-energy Poincar\'e section. Intersections of these manifolds can be accurately computed in just a few seconds as well, providing natural transfer trajectories between different mean motion resonances and demonstrating the usefulness of these parameterizations for space mission design and celestial dynamical analysis. 

\section*{Acknowledgements}

This work was supported in part by the National Science Foundation (NSF) under a Mathematical Sciences Postdoctoral Research Fellowship award no. DMS-2202994, and in part by the US Air Force Office of Scientific Research (AFOSR) under Award No. FA8655-24-1-7012. This research was carried out in part at the Jet Propulsion Laboratory, California Institute of Technology, under a contract with the National Aeronautics and Space Administration (80NM0018D0004).

\begin{appendices}
\section{Proof of constant $\bar \lambda_s, \bar \lambda_u$ with eigenvectors rescaled by $a_s, a_u$} \label{rescaleProof}

We prove the result for $\bold{\bar v}_{s}$; the case of $\bold{\bar v}_{u}$ can be proven in the exact same manner. 
\begin{lemma}
If $\bold{ v}_{s}(k)$, $\lambda_s(k)$ are such that $D_{\bold{x}}\Phi_{\tau(k)}(X(k)) \bold{v}_{s}(k) = \lambda_{s}(k) \bold{v}_{s}(k + 1 \mod m)$ for all $k=0, 1, \dots, m-1$, and $a_{s}(k)$,  $\bar \lambda_{s}$ satisfy Eq. \eqref{cohom_as}, then $\bold{\bar v}_{s}(k) = a_{s}(k) \bold{v}_{s}(k)$ satisfies
\begin{equation} D_{\bold{x}}\Phi_{\tau(k)}(X(k))  \bold{\bar v}_{s}(k) =  \bar \lambda_{s} \bold{\bar v}_{s}(k + 1 \mod m)\end{equation}
\end{lemma}
\begin{proof}
Since $D_{\bold{x}}\Phi_{\tau(k)}(X(k)) \bold{v}_{s}(k) = \lambda_{s}(k) \bold{v}_{s}(k + 1 \mod m)$, we have
\begin{align} \begin{split} D_{\bold{x}}\Phi_{\tau(k)}(X(k))  \bold{\bar v}_{s}(k) &= D_{\bold{x}}\Phi_{\tau(k)}(X(k))   a_{s}(k) \bold{ v}_{s}(k)   =a_{s}(k) \lambda_{s}(k) \bold{ v}_{s}(k + 1 \mod m) \\
&= a_{s}(k + 1 \mod m)  \bar \lambda_{s} \bold{v}_{s}(k + 1 \mod m)) =  \bar \lambda_{s}  \bold{\bar v}_{s}(k + 1 \mod m) \end{split} \end{align}
where $a_{s}(k) \lambda_{s}(k) = a_{s}(k + 1 \mod m) \bar \lambda_{s} $ follows from exponentiating Eq. \eqref{cohom_as}. \qed
\end{proof}

\section{Proof of properties of $\bold{\bar v}_2(k)$} \label{sympConjProof}

Here, we show that the vectors $\bold{\bar v}_{2}(k)$ found using the procedure of Section \ref{step1} satisfy Equation \eqref{sympConjEquation}
This is proven as a result of the two lemmas below, both adapted from similar results for invariant tori \cite{kumar2022}. 

\begin{lemma} \label{sympconjLemma}
The vectors $\bold{v}_{2}(k)$ defined in Eq. \eqref{prelimSympConj} satisfy  
\begin{equation}  \label{prelimSympConjCondition}
D_{\bold{x}}\Phi_{\tau(k)}(X(k))\bold{v}_{2}(k) = T(k) \bold{\bar v}_{1}(k + 1 \mod m) + \bold{v}_{2}(k + 1 \mod m)\end{equation}
\end{lemma}

\begin{lemma} \label{constTLemma}
If $\bold{v}_{2}(k)$ and $a(k)$ satisfy Eq. \eqref{prelimSympConjCondition} and \eqref{Tkill} respectively with $\bar T = \frac{1}{m}\sum_{k=0}^{m-1} T(k) $ a constant independent of $k$, then $\bold{\bar v}_{2}(k) = \bold{v}_{2}(k) + a(k) \bold{\bar v}_{1}(k)$ will satisfy Equation \eqref{sympConjEquation}:
\begin{equation}  D_{\bold{x}}\Phi_{\tau(k)}(X(k))  \bold{\bar v}_{2}(k) =  \bar T \bold{\bar v}_{1}(k + 1 \mod m)  + \bold{\bar v}_{2}(k + 1 \mod m)   \end{equation}
\end{lemma}

\begin{proof}[Proof of Lemma \ref{sympconjLemma}]
Combining Eq. \eqref{prelimSympConj} with Eq. \eqref{abcd} and the known relations $D_{\bold{x}}\Phi_{\tau(k)}(X(k)) \bold{\bar v}_{s}(k) = \bar \lambda_s \bold{\bar v}_{s}(k+1 \mod m) $ and $D_{\bold{x}}\Phi_{\tau(k)}(X(k)) \bold{\bar v}_{u}(k) = \bar \lambda_u \bold{\bar v}_{u}(k+1 \mod m) $, we have 
\begin{align} \begin{split} D_{\bold{x}}\Phi_{\tau(k)}(X(k))&\bold{v}_{2}(k) = D_{\bold{x}}\Phi_{\tau(k)}(X(k))\left( \frac{J^{-1} \bold{\bar v}_{1}(k)}{ \|\bold{\bar v}_{1}(k)\|^{2}} + f_{1}(k) \bold{\bar v}_{s}(k)  + f_{2}(k) \bold{\bar v}_{u}(k) \right) \\
= &T(k) \bold{\bar v}_{1}(k + 1 \mod m) + B(k) \frac{J^{-1} \bold{\bar v}_{1}(k + 1 \mod m)}{ \|\bold{\bar v}_{1}(k + 1 \mod m)\|^{2}} \\ 
&+ \left(C(k) + \bar \lambda_{s}f_{1}(k) \right) \bold{\bar v}_{s}(k + 1 \mod m) + \left( D(k) + \bar \lambda_{u}f_{2}(k) \right)\bold{\bar v}_{u}(k + 1 \mod m) 
\end{split} \end{align} 
Recalling Equations \eqref{cEquation} and \eqref{dEquation} defining $f_1$ and $f_2$, we thus have that
\begin{align} \begin{split} \label{DFtimesvc} D_{\bold{x}}\Phi_{\tau(k)}(X(k))\bold{v}_{2}(k) = &T(k) \bold{\bar v}_{1}(k + 1 \mod m) + B(k) \frac{J^{-1} \bold{\bar v}_{1}(k + 1 \mod m)}{ \|\bold{\bar v}_{1}(k + 1 \mod m)\|^{2}} \\ 
&+ f_{1}(k + 1 \mod m)  \bold{\bar v}_{s}(k + 1 \mod m) + f_{2}(k + 1 \mod m) \bold{\bar v}_{u}(k + 1 \mod m) 
\end{split} \end{align} 
Denote the bilinear symplectic form $\Omega$ on Euclidean $\mathbb{R}^{4}$ as $ \Omega(\bold{v}_{1}, \bold{v}_{2}) =  \bold{v}_{1}^{T} J \bold{v}_{2} $; it is easy to see that $\Omega(\bold{v}_{1}, \bold{v}_{1}) = 0$ for any $\bold{v}_{1} \in \mathbb{R}^{4}$. As flow maps of Hamiltonian systems are symplectic \cite{thirring}, we have that $\Omega(\bold{v}_{1}, \bold{v}_{2}) = \Omega(D_{\bold{x}}\Phi_{\tau(k)}(X(k))\bold{v}_{1}, D_{\bold{x}}\Phi_{\tau(k)}(X(k))\bold{v}_{2})$ for all $\bold{v}_{1}$, $\bold{v}_{2} \in \mathbb{R}^{4}$ and $k=0, 1, \dots, m-1$.  Thus, 
\begin{align} \begin{split} \label{deriveZero} \max_{k} |\Omega(\bold{\bar v}_{1}&(k),\bold{\bar v}_{s}(k)) |= \max_{k} \left| \Omega \left(D_{\bold{x}}\Phi_{\tau(k)}(X(k))\bold{\bar v}_{1}(k), D_{\bold{x}}\Phi_{\tau(k)}(X(k)) \bold{\bar v}_{s}(k) \right) \right| \\
&=\max_{k}  \left| \Omega \left(\bold{\bar v}_{1}(k + 1 \mod m), \bar \lambda_s \bold{\bar v}_{s}(k + 1 \mod m) \right) \right| \\
&= |\bar \lambda_s|  \max_{k}  \left| \Omega \left(\bold{\bar v}_{1}(k + 1 \mod m), \bold{\bar v}_{s}(k + 1 \mod m) \right) \right| = |\bar \lambda_s|  \max_{k} \left| \Omega \left(\bold{\bar v}_{1}(k), \bold{\bar v}_{s}(k) \right) \right| \\
\end{split} \end{align} 
which implies that $\max_{k} \left| \Omega \left(\bold{\bar v}_{1}(k), \bold{\bar v}_{s}(k) \right) \right| = 0$ since $0 < |\bar \lambda_s| < 1$. Thus, for all $k=0, 1, \dots, m-1$, $\Omega \left(\bold{\bar v}_{1}(k), \bold{\bar v}_{s}(k) \right) =0$ . We can also show that all $\Omega \left(\bold{\bar v}_{1}(k), \bold{\bar v}_{u}(k) \right) = 0$ in a very similar manner. Hence, 
\begin{align} \label{areaOne} \begin{split} \Omega(\bold{\bar v}_{1}(k),\bold{v}_{2}(k)) &= \Omega \left(\bold{\bar v}_{1}(k),\frac{J^{-1} \bold{\bar v}_{1}(k)}{ \|\bold{\bar v}_{1}(k)\|^{2}} + f_{1}(k) \bold{\bar v}_{s}(k)  + f_{2}(k) \bold{\bar v}_{u}(k) \right) \\
&= \Omega \left(\bold{\bar v}_{1}(k),\frac{J^{-1} \bold{\bar v}_{1}(k)}{ \|\bold{\bar v}_{1}(k)\|^{2}} \right) \\
&= \bold{\bar v}_{1}(k)^{T}J \frac{J^{-1} \bold{\bar v}_{1}(k)}{ \|\bold{\bar v}_{1}(k)\|^{2}}  = \frac{ \bold{\bar v}_{1}(k)^{T} \bold{\bar v}_{1}(k)}{ \|\bold{\bar v}_{1}(k)\|^{2}} = 1
\end{split} \end{align} 
Since $D_{\bold{x}}\Phi_{\tau(k)}(X(k))$ are symplectic matrices, using Eq. \eqref{DFtimesvc} we have that 
\begin{align} \begin{split} \label{BisOne} 1 = \Omega &(\bold{\bar v}_{1}(k),\bold{v}_{2}(k)) \\
= \Omega &\left(D_{\bold{x}}\Phi_{\tau(k)}(X(k))\bold{\bar v}_{1}(k),D_{\bold{x}}\Phi_{\tau(k)}(X(k))\bold{v}_{2}(k) \right) \\
= \Omega &\left(\bold{\bar v}_{1}(k + 1 \mod m),T(k) \bold{\bar v}_{1}(k + 1 \mod m) + B(k) \frac{J^{-1} \bold{\bar v}_{1}(k + 1 \mod m)}{ \|\bold{\bar v}_{1}(k + 1 \mod m)\|^{2}} \right. \\
& \quad \quad \quad \quad \quad \quad \left. + f_{1}(k + 1 \mod m)  \bold{\bar v}_{s}(k + 1 \mod m) + f_{2}(k + 1 \mod m) \bold{\bar v}_{u}(k + 1 \mod m) \vphantom{\frac{J^{-1} \bold{\bar v}_{1}(k + 1 \mod m)}{ \|\bold{\bar v}_{1}(k + 1 \mod m)\|^{2}}} \right) \\
= \Omega &\left(\bold{\bar v}_{1}(k + 1 \mod m), B(k) \frac{J^{-1} \bold{\bar v}_{1}(k + 1 \mod m)}{ \|\bold{\bar v}_{1}(k + 1 \mod m)\|^{2}} \right) \\
=B(&k) \bold{\bar v}_{1}(k + 1 \mod m)^{T} J  \frac{J^{-1} \bold{\bar v}_{1}(k + 1 \mod m)}{ \|\bold{\bar v}_{1}(k + 1 \mod m)\|^{2}} = B(k)
\end{split} \end{align} 
proving that $B(k) =1$ for all $k=0, 1, \dots, m-1$. Therefore, substituting this into Eq. \eqref{DFtimesvc} gives
\begin{align} \begin{split} \label{subInto} D_{\bold{x}}\Phi_{\tau(k)}(X(k))\bold{v}_{2}(k) = &T(k) \bold{\bar v}_{1}(k + 1 \mod m) +  \frac{J^{-1} \bold{\bar v}_{1}(k + 1 \mod m)}{ \|\bold{\bar v}_{1}(k + 1 \mod m)\|^{2}} \\ 
&+ f_{1}(k + 1 \mod m)  \bold{\bar v}_{s}(k + 1 \mod m) + f_{2}(k + 1 \mod m) \bold{\bar v}_{u}(k + 1 \mod m) 
\end{split} \end{align} 
Finally, we see from Eq. \eqref{prelimSympConj} that the last 3 terms on the RHS of Eq. \eqref{subInto} are just $\bold{v}_{c}(k + 1 \mod m)$, so
\begin{equation}  D_{\bold{x}}\Phi_{\tau(k)}(X(k))\bold{v}_{2}(k) = T(k) \bold{\bar v}_{1}(k + 1 \mod m) + \bold{v}_{2}(k + 1 \mod m) \end{equation} 
which is what we sought to prove. 
\end{proof}

\begin{proof} [Proof of Lemma \ref{constTLemma}]
Since $\bold{v}_{2}(k)$ satisfies Eq. \eqref{prelimSympConjCondition} and $D_{\bold{x}}\Phi_{\tau(k)}(X(k)) \bold{\bar v}_{1}(k) = \bold{\bar v}_{1}(k + 1 \mod m)$, we have
\begin{align} \begin{split} D_{\bold{x}}\Phi_{\tau(k)}(X(k)) \bold{\bar v}_{2}(k) &= D_{\bold{x}}\Phi_{\tau(k)}(X(k)) \left[ \bold{v}_{2}(k) + a(k) \bold{\bar v}_{1}(k) \right] \\
&= \left[T(k)+a(k) \right] \bold{\bar v}_{1}(k + 1 \mod m)+\bold{v}_{2}(k + 1 \mod m) \\
&= \left[\bar T+a(k + 1 \mod m) \right] \bold{\bar v}_{1}(k + 1 \mod m)+\bold{v}_{2}(k + 1 \mod m) \\
&= \bar T \bold{\bar v}_{1}(k + 1 \mod m)+\bold{\bar v}_{2}(k + 1 \mod m) \end{split} \end{align}
where the relation $T(k)+a(k) = \bar T+a(k + 1 \mod m)$ follows from Eq. \eqref{Tkill}. 
\end{proof}

\section{Proof that $A,B$ are contraction maps in $\ell^\infty$ norm} \label{contractionProof}

We prove that $A$ is a contraction if $0<\alpha<1$; the same can be shown for $B$ and $\alpha > 1$ very similarly. 
\begin{lemma}
If $0<\alpha <1$, $A$ is a contraction map. Hence, in this case the iteration $u_{n+1} = A(u_{n})$ uniformly converges exponentially fast as $n \rightarrow \infty$ to the solution $u$ of Equation \eqref{u1contract} (and thus also \eqref{cohomHyp}). 
\end{lemma}
\begin{proof}
Let $u_{1},u_{2}$ be finite sequences indexed by $k=0, 1, \dots, m-1$. Then, 
\begin{align} \begin{split}  \max_{k}  \| [A(&u_1)](k) - [A(u_2)](k) \| \\
&=  \max_{k} \| \alpha u_1(k-1 \mod m) -  \alpha u_2(k-1 \mod m)  \|  \\
&=  \alpha \max_{k} \| u_1(k-1 \mod m) - u_2(k-1 \mod m)  \| = \alpha \max_{k} \| u_1(k) - u_2(k) \| 
\end{split} \end{align} 
As $0<\alpha<1$, $A$ is a contraction map under the $\ell^\infty$ norm. The contraction mapping theorem \citep{chicone2006} tells us that every such map has a unique fixed point; furthermore, the fixed point can be found by iterating any value in the domain of the map forwards until convergence. The solution of Equations \eqref{u1contract}  is by definition the fixed point of contraction map $A$. Hence, the iteration converges to $u$. 
\end{proof}

\end{appendices}

\bibliography{bibfile}
\bibliographystyle{sn-mathphys-num}

\end{document}